\pdfoutput=1
\RequirePackage{snapshot}
\makeatletter
\def\snap@providesfile#1[#2]{%
  \wlog{File: #1 #2}%
  \if\expandafter\snap@graphic@test\expanded{#2}@@\@nil
    \snap@record@graphic#1\relax #2 (type ??)\@nil
  \else
    \expandafter\xdef\csname ver@#1\endcsname{#2}%
  \fi
  \endgroup
}
\makeatother

\documentclass[a4paper,11pt,twoside,
english,fabfinal]{fabsmfart}

\usepackage{fabmacromath,fababbreviations,fabalphabets,fabmaps,fabQFT,fabcolgraphs}

\usepackage{soul,xparse,pifont}
\setstcolor{red}

\usepackage{tikz}
\usetikzlibrary{positioning,matrix}
\usetikzlibrary{arrows,positioning,calc,decorations.pathmorphing}
\tikzset{
  vertex/.style={circle, draw=white,
    fill=black!50, thick, inner sep=0pt, minimum size=1.3mm},
  vertexbb/.style={circle, draw=black,
    fill=black!40, thin, inner sep=0pt, minimum size=1.1mm,
    opacity=1},
  vertexbw/.style={circle, draw=black,
    fill=white, thin, inner sep=0pt, minimum size=1.1mm, opacity=1},
  lab/.style={circle,fill=white,outer sep=2pt,inner
    sep=0pt,auto,midway,scale=.6},
  lab2/.style={circle,fill=white,outer sep=4pt,inner sep=0pt,midway,auto,scale=.6},
  arete/.style={
    ultra thick},
  presimp/.style={white, thick},
  simp/.style={gray, thin, opacity=1},
  behind/.style={gray, thick, opacity=.3},
  prep/.style={line width=3pt, white},
  main/.style={decoration={random
      steps,segment length=5pt,amplitude=1pt}},
  maing/.style={decoration={random
    steps,segment length=6pt,amplitude=2pt}},
}
\usepackage{tikz-3dplot}
\usepackage{xypic}

\numberwithin{equation}{section}

\usepackage[hypertexnames=false]{hyperref}
\hypersetup{%
 pdftitle={Partial Duality of Hypermaps},%
  pdfauthor={Sergei Chmutov, Fabien Vignes-Tourneret},%
  pdfsubject={},%
  pdfkeywords={},%
 }
\usepackage[level=2]{bookmark}

\usepackage{graphbox,pbox}
\graphicspath{{figures/}{./}}
\definecolor{zeroS}{HTML}{ff0000}
\definecolor{oneS}{HTML}{009000}
\definecolor{twoS}{HTML}{0000ff}

\usepackage{caption}
\DeclareCaptionStyle{smf}[justification=centering]{width=.85\textwidth,format=plain,labelformat=simple,labelsep=period,justification=justified,font={normalfont,small},labelfont={normalfont,bf,small},textfont+={up},singlelinecheck=true}
\usepackage{subcaption}

\usepackage[capitalise,nameinlink]{cleveref}

\usepackage[backend=bibtex, backref=false, style=alphabetic,giveninits=true]{biblatex} 
\usepackage{enumitem}
\setlist{leftmargin=*,labelindent=17.0pt}

\allowdisplaybreaks[1]

\usepackage{commath}

\newcommand{\ig}{\includegraphics}
\renewcommand{\rb}{\raisebox}

\newcommand\risSpdf[6]{\rb{#1pt}[#5pt][#6pt]{\begin{picture}(#4,15)(0,0)
  \put(0,0){\ig[width=#4pt]{#2.pdf}} #3
     \end{picture}}}

\title{Partial Duality of Hypermaps}
\shorttitle{Partial Duality of Hypermaps}
\author{S.~Chmutov \& F.~Vignes-Tourneret}
\dedicatory{}

\begin{document}
\maketitle

\begin{fabsmfabstract}
We introduce partial duality of hypermaps, which include the classical
Euler-Poincaré duality as a particular case. Combinatorially,
hypermaps may be described in one of three ways: as three involutions
on the set of flags (bi-rotation system or $\tau$-model), or as three
permutations on the set of half-edges (rotation system or
$\sigma$-model in orientable case), or as edge 3-coloured graphs. We
express partial duality in each of these models. We give a formula for the genus change under partial duality.
\end{fabsmfabstract}
\begin{fabsmfMSC}
  05C10, 05C65, 57M15, 57Q15
\end{fabsmfMSC}
\begin{fabsmfkeywords}
  maps, hypermaps, partial duality, permutational models, rotation system, bi-rotation system, edge coloured graphs
\end{fabsmfkeywords}

\vfil
\tableofcontents
\newpage

\section*{Introduction}
\label{intro}
\etoctoccontentsline*{section}{Introduction}{1}

\emph{Maps} can be thought of as graphs embedded into surfaces. \emph{Hypermaps} are hypergraphs embedded into surfaces. In other words, in hypermaps a (hyper) edge is allowed to connect more than two vertices, so having more than two \emph{half-edges}, or just a single half-edge (see Figure \ref{f-LocalHypermap}).

One way of combinatorially study oriented hypermaps, the
\emph{rotation system} or the $\sigma$-model, is to consider
permutations of its half-edges, also know as \emph{darts}, around each
vertex, around each hyperedge, and around each face according to the
orientation. This model has been carefully worked out by R.~Cori
\cite{Cori75}, however it can be traced back to L.~Heffter
\cite{Heffter1891aa}. It became popular after the work of
J.~R.~Edmonds \cite{Edmonds1960aa}. It is very important for the Grothendieck dessins d'Enfants theory, see \cite{Lando2004aa}, where the $\sigma$-model is called \emph{$3$-constellation}. We review the $\sigma$-model in \cref{ssec-sigma-model}.

Another combinatorial description of hypermaps, the \emph{bi-rotation
  system} or the $\tau$-model, goes through three involutions acting
on the set of \emph{local flags}, also know as \emph{blades},
represented by triples (vertex, edge, face). The motivation for this
model was the study of symmetry of regular polyhedra which is a group
generated by reflections (involutions). As such it may be traced back
to Ancient Greeks. It was used systematically by F.~Klein in
\cite{Kelin1956aa} and later by Coxeter and Moser in
\cite{Coxeter1980aa}. More recently this model was used in the context
of maps and hypermaps in \cite{Wilson1978aa,Jones1983aa,James1988aa,Jones2009aa}. We review the $\tau$-model in \cref{ssec-tau-model}.

In 1975 T.~Walsh noted \cite{Walsh1975aa} that, if we consider a small
regular neighbourhood of vertices and hyperedges, then we can regard
hypermaps as cell decomposition of a compact closed surface into disks
of three types, vertices, hyperedges, and faces, such that the disks
of the same type do not intersect and the disks of different types may
intersect only on arcs of their boundaries. These arcs form a
3-regular graph whose edges are coloured in 3 colours depending on the
types of cells they are adjacent to. The arcs of intersection of
hyperedge-disks with face-disks bear the colour 0. The colour 1 stands
for the arcs of intersection of vertex-disks with face-disks. And the
arcs of intersection of vertex-disks with hyperedge-disks are coloured
by 2. Thus we come to the concept of $[2]$-coloured graphs, where $[2]$
stands for the set of three colours $[2]:=\{0,1,2\}$. It turns out that
such a $[2]$-coloured graph carries all the information about the
original hypermap. This gives another combinatorial model for
description of hypermaps. We review this model in
\cref{ssec-edge-coloured-graphs}.
\tdplotsetmaincoords{65}{10}
  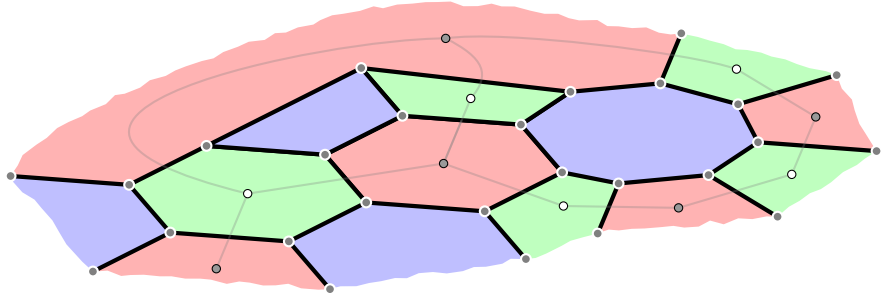
\begin{figure}[!htp]
    \centering
    \begin{tikzpicture}[tdplot_main_coords,scale=.8]
      \def\u{2}; 
      \colorlet{V}{red!30};
      \colorlet{E}{green!50!white!50};
      \colorlet{F}{blue!50!white!50};
      \tdplotsetcoord{a}{\u}{90}{0};
      \tdplotsetcoord{b}{\u}{90}{60};
      \tdplotsetcoord{c}{\u}{90}{120};
      \tdplotsetcoord{d}{\u}{90}{180};
      \tdplotsetcoord{e}{\u}{90}{240};
      \tdplotsetcoord{f}{\u}{90}{300};

      \tdplotsetcoord{h}{2*\u}{90}{120};
      \tdplotsetcoord{i}{2*\u}{90}{180};
      \tdplotsetcoord{j}{2*\u}{90}{240};
      \tdplotsetcoord{k}{2*\u}{90}{300};

      \coordinate (x) at ($(b)+(\u,0)$); 
      \coordinate (g) at ($(x)+(135:\u)$);
      \coordinate (m) at ($(x)+(90:\u)$);
      \coordinate (n) at ($(x)+(45:\u)$);
      \coordinate (o) at ($(x)+(0:\u)$);
      \coordinate (p) at ($(x)+(-45:\u)$);
      \coordinate (l) at ($(x)+(-90:\u)$);

      \coordinate (y) at ($(e)+(-\u,0)$);
      \coordinate (q) at ($(y)+(180:\u)$);
      \coordinate (r) at ($(y)+(240:\u)$);
      \coordinate (s) at ($(q)+(-\u,0)$);
      \coordinate (t) at ($(y)+(240:2*\u)$);
      \coordinate (u) at ($(y)+(300:2*\u)$);
      \coordinate (v) at ($(x)+(-90:2*\u)$);
      \coordinate (w) at ($(x)+(-45:2*\u)$);
      \coordinate (z) at ($(x)+(0:2*\u)$);
      \coordinate (aa) at ($(x)+(45:2*\u)$);
      \coordinate (ab) at ($(x)+(90:2*\u)$);

      \coordinate (ca) at (-30:1.3*\u);
      \coordinate (cb) at ($(e)+(-60:\u)$);
      \coordinate (cc) at ($(x)+(-45/2:1.5*\u)$);
      \coordinate (cd) at ($(x)+(-3*45/2:1.5*\u)$);
      \coordinate (ce) at ($(x)+(45/2:1.5*\u)$);
      \coordinate (cf) at ($(x)+(3*45/2:1.5*\u)$);
      \coordinate (cg) at ($(y)+(0,-1.5*\u)$);
      \coordinate (ch) at ($(y)+(210:1.5*\u)$);
      \coordinate (ci) at (90:1.3*\u);
      \coordinate (cj) at (150:1.3*\u);
      \coordinate (ck) at ($(ci)+(110:1.2*\u)$);

      \tdplotsetcoord{ac}{3*\u}{90}{120};

      \fill[E] (m) -- (n) -- (aa) decorate [main] {-- (ab)} -- cycle;
      \fill[V] (n) -- (aa) decorate [main] {-- (z)} -- (o) -- cycle;
      \fill[E] (z) decorate [main] {--(w)} --(p)--(o)--cycle; \fill[V]
      (p)--(w) decorate [main] {--(v)} --(l)--cycle; \fill[E] (v)
      decorate [main]{--(k)}--(f)--(a)--(l)--cycle; \fill[F] (k)
      decorate [main]{--(u)}--(j)--(e)--(f)--cycle; \fill[V] (u)
      decorate [main]{--(t)}--(r)--(j)--cycle; \fill[F] (t) decorate
      [main]{--(s)}--(q)--(r)--cycle; \fill[V] (s) decorate
      [main]{.. controls ($(s)+(0,\u)$) and ($(ac)+(-\u,-\u)$)..(ac)
        .. controls ($(ac)+(\u,\u)$) and
        ($(ab)+(-\u,0)$)..(ab)}--(m)--(g)--(h)--(i)--(q)--cycle;

      \draw[arete] (m)--(ab); \draw[arete] (n)--(aa); \draw[arete]
      (o)--(z); \draw[arete] (p)--(w); \draw[arete] (l)--(v);
      \draw[arete] (f)--(k); \draw[arete] (j)--(u); \draw[arete]
      (r)--(t); \draw[arete] (q)--(s);
   
      \filldraw[arete, fill=V] (a) -- (b) -- (c) -- (d) -- (e) -- (f)
      -- cycle; \filldraw[arete, fill=F] (g) -- (m) -- (n) -- (o) --
      (p) -- (l) -- (a) -- (b) -- cycle; \filldraw[arete, fill=E] (e)
      -- (d) -- (i) -- (q) -- (r) -- (j) -- cycle; \filldraw[arete,
      fill=E] (b) -- (c) -- (h) --(g) -- cycle; \filldraw[arete,
      fill=F] (c) -- (h) -- (i) -- (d) -- cycle;
  
      \draw[behind] (0,0)--(ci);
      \draw[behind] (0,0)--(cy);
      \draw[behind] (0,0)--(ca);
      \draw[behind] (0,0)--(y);
      \draw[behind, tension=1] (ci).. controls ($(ci)+(80:1)$) and
      ($(ck)+(-40:1)$) ..(ck);
      \draw[behind, tension=.2] (y).. controls ($(y)+(160:3*\u)$) and
      ($(ck)+(190:\u)$) ..(ck);
      \draw[behind] (ck).. controls ($(ck)+(30:1)$) and
      ($(cf)+(150:1)$) ..(cf);
      \draw[behind] (cf)--(ce);
      \draw[behind] (cc)--(ce);
      \draw[behind] (cc)--(cd);
      \draw[behind] (ca)--(cd);
      \draw[behind] (y)--(cg);

      \node[vertex] at (a) {};
      \node[vertex] at (b) {};
      \node[vertex] at (c) {};
      \node[vertex] at (d) {};
      \node[vertex] at (e) {};
      \node[vertex] at (f) {};
      \node[vertex] at (g) {};
      \node[vertex] at (h) {};
      \node[vertex] at (i) {};
      \node[vertex] at (j) {};
      \node[vertex] at (k) {};
      \node[vertex] at (l) {};
      \node[vertex] at (m) {};
      \node[vertex] at (n) {};
      \node[vertex] at (o) {};
      \node[vertex] at (p) {};
      \node[vertex] at (q) {};
      \node[vertex] at (r) {};
      \node[vertex] at (s) {};
      \node[vertex] at (t) {};
      \node[vertex] at (u) {};
      \node[vertex] at (v) {};
      \node[vertex] at (w) {};
      \node[vertex] at (z) {};
      \node[vertex] at (aa) {};
      \node[vertex] at (ab) {};
  
      \node[vertexbb] at (0,0) {};
      \node[vertexbw] at (y) {};
      \node[vertexbw] at (ca) {};
      \node[vertexbw] at (cc) {};
      \node[vertexbb] at (cd) {};
      \node[vertexbb] at (ce) {};
      \node[vertexbw] at (cf) {};
      \node[vertexbb] at (cg) {};
      \node[vertexbw] at (ci) {};
      \node[vertexbb] at (ck) {};
    \end{tikzpicture}
    \captionsetup{width=.75\textwidth}
    \caption{Local view of a hypermap with its Walsh map
      superimposed \\{\footnotesize (vertices are red, hyperedges are green, faces are blue).}}
\label{f-LocalHypermap}
\end{figure}

About the same time this concept was generalized to higher  dimensions.
Namely, in the $1970$'s M.~Pezzana \cite{Pezzana1974aa,Pezzana1975aa}
discovered a way of coding a piecewise-linear (PL) manifold by a
properly edge-coloured graph. The idea goes as follows: choose a triangulation $K$ of this given
manifold $M$. Consider then its first barycentric subdivision $K_{1}$. The
$1$-skeleton of $K_{1}^{*}$ is a properly edge-colourable graph. It
turns out that the colouring of the graph is sufficient to reconstruct $M$
completely. 
The discovery of M.~Pezzana allows to bring combinatorial and graph
theoretical methods into PL topology. This correspondence between PL
manifolds and coloured graphs has been further developed by M.~Ferri,
C.~Gagliardi and their group \cite{Ferri1986aa}. It has also been independently
rediscovered by A.~Vince \cite{Vince1983aa},
S~Lins and A.~Mandel \cite{Lins1985aa}, and, to a certain extent,
by R. Gurau \cite{Gurau2011aa}.

Originally the partial duality relative to a subset of edges was
defined for ribbon graphs in \cite{Chmutov2007aa} under the name of
generalized duality. The motivation came from an idea to unify various
versions of the Thistlethwaite theorems in knot theory relating the
Jones polynomial of knots with the Tutte-like polynomial of (ribbon)
graphs. Then it was thoughtfully studied and developed in papers
\cite{Vignestourneret2008aa,Moffatt2008aa,Moffatt2009aa,Ellis-Monaghan2010aa,Bradford2012aa,Moffatt2012aa,Huggett2013aa,Gross2020aa}. We
refer to \cite{Ellis-Monaghan2013aa} for an excellent account on this development.

The main result of this paper is a generalization of partial duality to hypermaps in \cref{sec-pdual}. There we define the partial duality
in \cref{ssec-definition} and then describe it in each of the  three
combinatorial models in subsequent subsections. Independently this
generalization was found by Benjamin Smith \cite{Smith2018aa}, but it does not contain the expression of partial duality in terms of permutational models and does not have any formula for the genus change. The operation of partial duality usually is different from the operations of \cite{Jones2009aa,Jones1983aa} and from the operation of \cite{Vince1995aa}. Typically it changes the genus of a hypermap. We give a formula for the genus change in \cref{sec-genus}.
We finish the paper with general remarks about future directions of
research on partial duality in higher dimensions.

\paragraph{Acknowledgement}
\label{sec-acknowledgement}

We are grateful to Iain Moffatt and Neal Stoltzfus for fruitful discussions on our preliminary results during the Summer 2014 Programm ``Combinatorics, geometry, and physics'' in Vienna, and the Erwin Schr\"odinger International Institute for Mathematical Physics (ESI) and the University of Vienna for hospitality during the program. 
S.~Ch. thanks the Max-Plank-Institut f\"ur Mathematik in Bonn and the
Universit\'e Lyon 1 for excellent working conditions and warm
hospitalities during the visits in 2014 and 2016. F.~V.-T.\@ is
partially supported by the grant ANR JCJC CombPhysMat2Tens.

\section{Hypermaps}
\label{sec-hypermaps}

\subsection{Geometrical model}
\label{ssec-geom-model}

A {\it\bfseries map} is a cellularly embedded graph in a (not necessarily orientable) compact closed surface. The edges of a graph are represented by smooth arcs on the surface connecting two (not necessarily distinct) vertices. A small regular neighbourhood of such a graph on the surface is a surface with boundary, called 
{\it\bfseries ribbon graph}, equipped with a decomposition into a union of topological disks of two types, the neighbourhoods of vertices and the neighbourhoods of edges. The last one can be regarded as a narrow quadrilateral along the edges attached to the corresponding vertex discs at the two opposite sides. Attaching disks called {\it\bfseries faces} to the boundary components of a ribbon graph restores the original closed surface. Thus a map may be regarded as a cell decomposition of a compact closed surface into disks of three types, vertices, edges, and faces, such that the disks of the same type do not intersect and the disks of different types may intersect only on arcs of their boundaries and the edge-disks intersect with at most two vertex-disks and at most two face-disks.

{\it\bfseries Hypermaps} differ from maps in that the edges are allowed to be 
{\it\bfseries hyperedges} and may connect several vertices. Hypermaps may be considered as cellularly embedded hypergraphs where hyperedges are embedded as one dimensional star-shapes. 
Consequently hypermaps can be defined as cell decomposition of a compact closed surface into disks of three types, vertices, hyperedges, and faces, such that the disks of the same type do not intersect and the disks of different types may intersect only on arcs of their boundaries. The restriction on edge-disks to be quadrilaterals is released here comparable to the definition of a map.
So the definition of a hypermap is completely symmetrical with respect to the types of the cells.
\begin{figure}[!htp]
  \centering
  \begin{subfigure}[c]{.3\linewidth}
    \centering
    \raisebox{-.5\height}{\includegraphics[scale=.35]{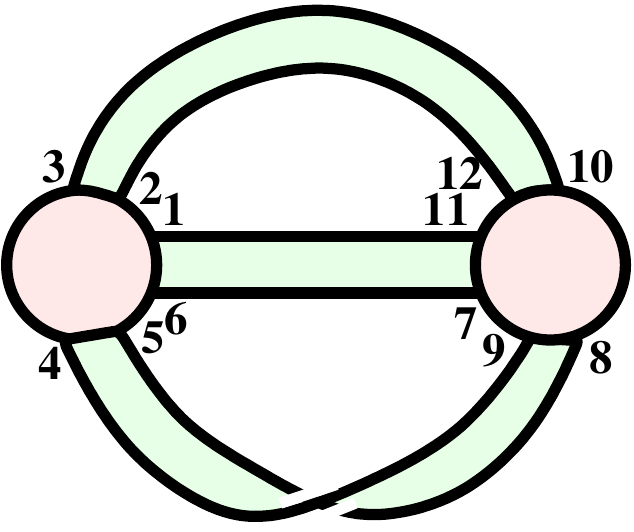}}
    \caption{A non orientable map $\mapm_0$ on a projective plane with
      two vertices and three edges.}
    \label{fig-NOMapEx-left}
  \end{subfigure}\qquad\qquad
  \begin{subfigure}[c]{.55\linewidth}
    \centering
    \raisebox{-.5\height}{\includegraphics[scale=.35]{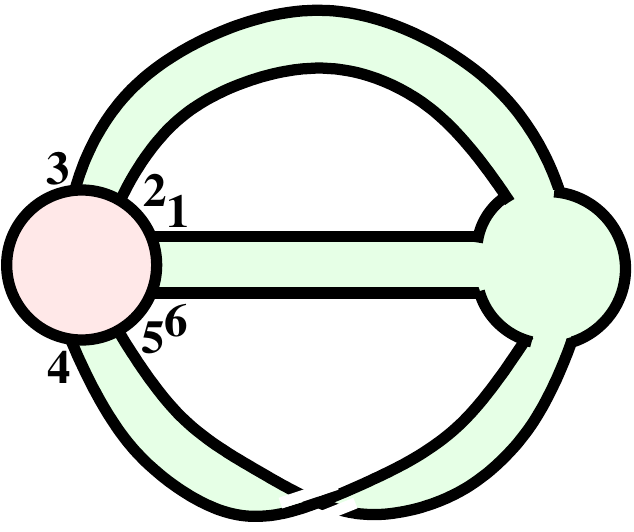}}\quad
    $=$\quad
    \raisebox{-.5\height}{\includegraphics[scale=.4]{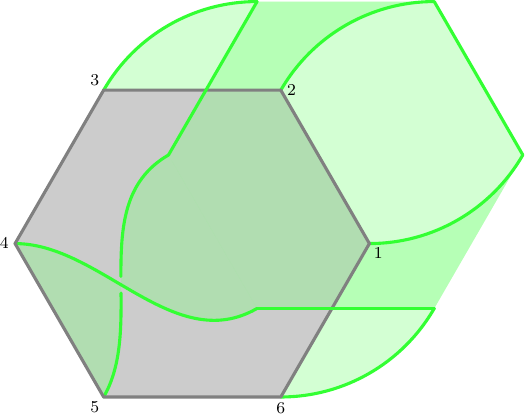}}
    \caption{A non orientable hypermap ${\mathfrak{hm}}_0$ on a
      projective plane with one vertex and one hyerpedge.}
    \label{fig-NOMapEx-right}
  \end{subfigure}
  \caption{Map and hypermap}
  \label{fig:NOMapEx}
\end{figure}
\Cref{fig:NOMapEx} shows a non orientable map $\mapm_{0}$ and a hypermap ${\mathfrak{hm}}_0$ obtained from $\mapm_{0}$ by uniting the right vertex with the three edges into a single hyperedge.  

\subsection{Permutational $\mathbf{\tau}$-model}
\label{ssec-tau-model}

In this model, also called {\it\bfseries bi-rotation system},
a hypermap $\map{hm}$ is described in a pure combinatorial way as three fixed point free involutions, $\tau_0$, $\tau_1$, and $\tau_2$,  acting on a set $X$ of 
{\it\bfseries local flags} of $\map{hm}$.
A (local) flag is a triple $(v,e,f)$ consisting of a vertex $v$, the intersection $e$ of a hyperedge incident to $v$ with a small neighbourhood of $v$, the intersection $f$ of a face adjacent to $v$ and $e$ with the same neighbourhood of $v$. Another way of defining a local flag is to consider a triangle in the barycentric subdivision of faces of $\map{hm}$ considered as an embedded hypergraph. We will depict a flag as a small black copy of such a triangle attached to the vertex $v$.
\begin{figure}[!htp]
  \centering
  \includegraphics[scale=.4]{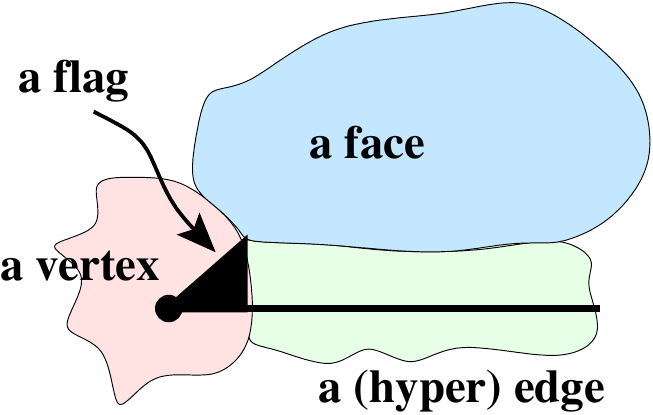}
  \caption{A local flag}
  \label{fig:flag}
\end{figure}

When a hypermap is understood as a [2]-coloured cell decomposition of a surface, the local flags correspond to the points where all three types of cells meet together.  Three lines of cell intersections emanate from each such point, the $2$-line of intersection of the vertex-disk with the (hyper) edge-disk, the $1$-line of intersection of the vertex-disk with the face-disk,  and the $0$-line of intersection of the edge-disk with the face-disk. These lines yield three partitions of the set $X$ of local flags into pairs of flags whose corresponding points are connected  by $0$-, $1$-, or $2$-lines. The permutation $\tau_i$ swaps the flags in the pairs connected by the $i$-lines. 
\begin{figure}[!htp]
  \centering
  \includegraphics[scale=.4]{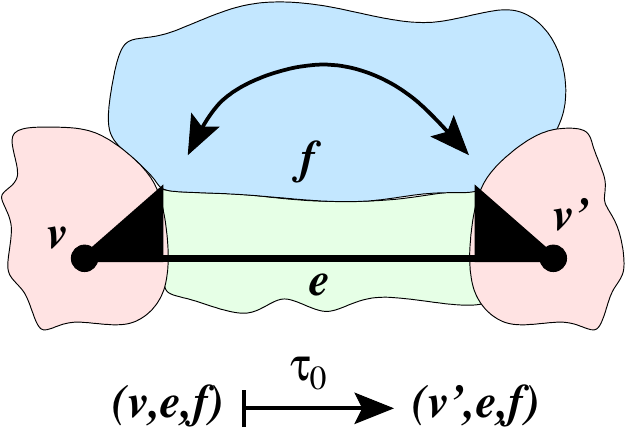}\qquad
	\includegraphics[scale=.4]{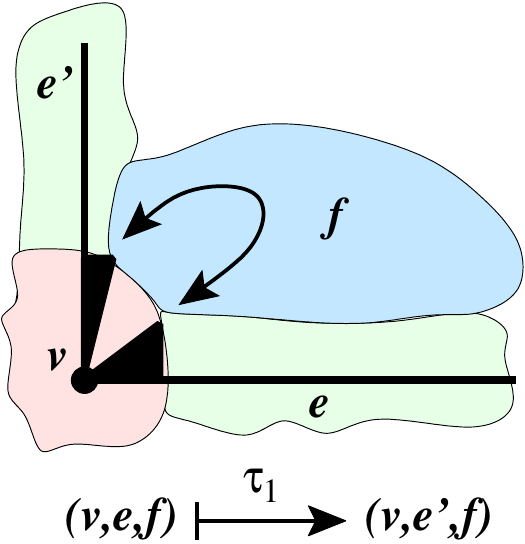}\qquad
	\includegraphics[scale=.4]{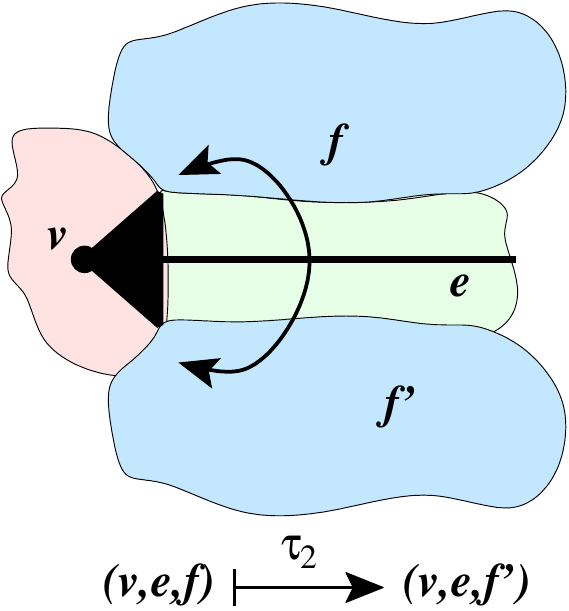}
  \caption{Involutions $\tau_0$, $\tau_1$, $\tau_2$}
  \label{fig:tau-0-1-2}
\end{figure}

In \cref{fig:NOMapEx} the local flags are labeled by
numbers. For these hypermaps the permutations $\tau_i$ are the
following. For the map $\mapm_{0}$,
\begin{align*}
  \tau_0&=(1,11)(2,12)(3,10)(4,9)(5,8)(6,7),\\
  \tau_1&=(1,2)(3,4)(5,6)(7,9)(8,10)(11,12),\\
  \tau_2&=(1,6)(2,3)(4,5)(7,11)(8,9)(10,12).
\end{align*}
For ${\mathfrak{hm}}_0$, 
\begin{align*}
  \tau_0&=(1,2)(3,5)(4,6),\quad\tau_1=(1,2)(3,4)(5,6),\quad\tau_2=(1,6)(2,3)(4,5).
\end{align*}

Any three fixed-point free involutions on a set $X$ yield a hypermap. Its vertices correspond to orbits of the subgroup $\langle\tau_1,\tau_2\rangle$ generated by $\tau_1$ and $\tau_2$, 
edges to the orbits of  $\langle\tau_0,\tau_2\rangle$, and faces to the orbits of 
$\langle\tau_0,\tau_1\rangle$. A hyperedge is a genuine edge if the corresponding orbit consists of four elements. Thus a hypermap is a map if and only if  $\tau_0\tau_2$ is also an involution.
\begin{rem}
W.~Tutte \cite{Tutte1984aa} introduced a less symmetrical description of combinatorial maps in terms of three permutations $\theta$, $\phi$, and $P$. They can be expressed in terms of 
$\tau_0$, $\tau_1$, and $\tau_2$ as follows:
\begin{equation*}
\theta=\tau_2,\quad \phi=\tau_0,\quad P=\tau_1\tau_2.
\end{equation*}
\end{rem}

\subsection{Permutational $\mathbf{\sigma}$-model}
\label{ssec-sigma-model}

This model, also known as {\it\bfseries rotation system}, gives a presentation of an  oriented hypermap in terms of three permutations 
$\sigma_V$, $\sigma_E$, and $\sigma_F$ of its {\it half-edges} $H$ satisfying the relation
$\sigma_F\sigma_E\sigma_V=1$. A {\it half-edge} is a small part of a hyperedge near a vertex incident to this hyperedge. We may think of half-edges as non complete local flags $(v,e)$ consisting of a vertex $v$ and the intersection $e$ of a hyperedge incident to $v$ with a small neighbourhood of  $v$. So a genuine edge has two half-edges, but a hyperedge may have more than two half-edges or even a single half-edge. 
If we think of a hyperedge as a one-dimensional star embedded to the surface, then
the rays of the star are the half-edges of the hyperedge. We place an empty square at the center of the star in order to distinguish it from a vertex.
 
The permutation $\sigma_V$ is a cyclic permutation of half-edges incident to a vertex according to the orientation of the hypermap. The permutation $\sigma_E$ acts as the cyclic permutation of rays in each star according to the orientation. For the permutation $\sigma_F$ we need to direct the half-edges with arrows pointing away from the vertices to which they are attached. These arrows point toward the centers of the stars of the hyperedges. The permutation $\sigma_F$ cyclically  permutes those half-edges in each face which are directed along the orientation of the face.
\begin{figure}[!htp]
  \centering
  \includegraphics[scale=.4]{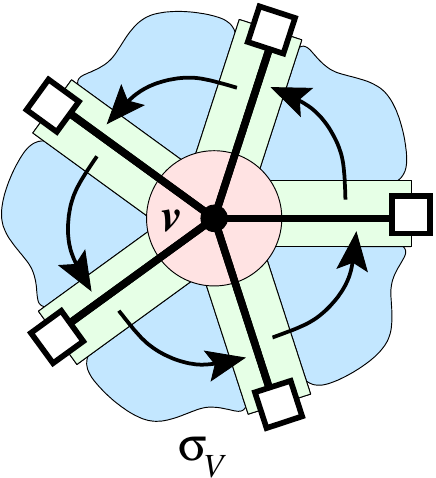}\qquad
  \includegraphics[scale=.4]{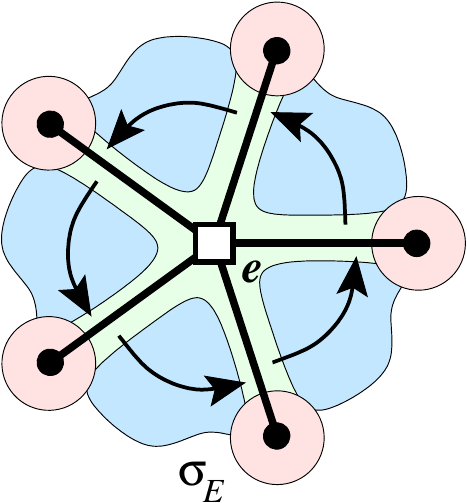}\qquad
  \includegraphics[scale=.4]{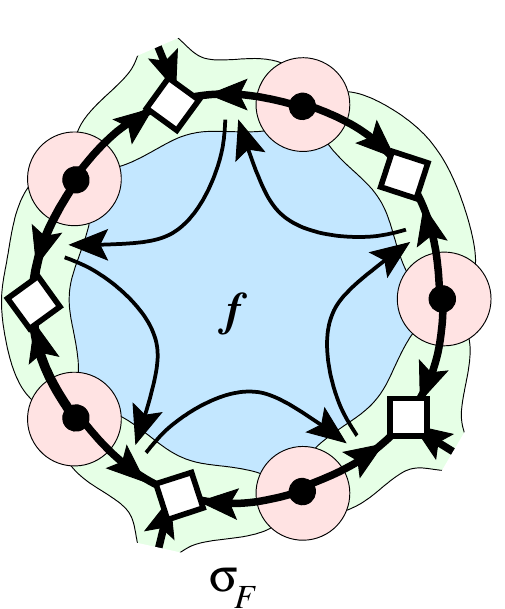}\qquad
  \raisebox{20pt}{\includegraphics[scale=.4]{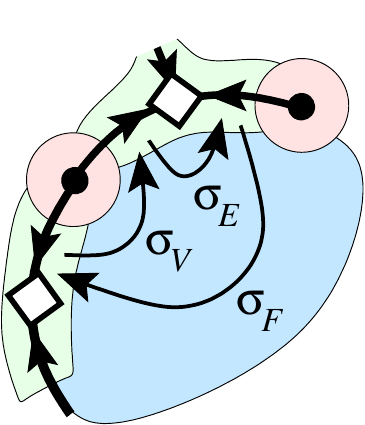}}
\caption{Permutations $\sigma_V$, $\sigma_E$, $\sigma_F$, and the identity $\sigma_F\sigma_E\sigma_V=1$}
  \label{fig:sigma_V_E_F}
\end{figure}
One can easily check that $\sigma_F\sigma_E\sigma_V=1$, see Figure \ref{fig:sigma_V_E_F}.
The cycles of $\sigma_V$ correspond to the vertices of the hypermap, the cycles of $\sigma_E$ correspond to the hyperedges, and the cycles of $\sigma_F$ correspond to the faces of the hypermap. Consequently any three permutations $\sigma_V$, $\sigma_E$, and $\sigma_F$ of a set $H$ satisfying the relation $\sigma_F\sigma_E\sigma_V=1$ uniquely determines an oriented hypermap.\\

Now let us describe the relation with the $\tau$-model of \cref{ssec-tau-model}. 
Each half-edge has two local flags in which it participates, if $x\in X$ is one of them, then $\tau_2(x)$ is the other one. Therefore the cardinality of $H$ is twice smaller that the cardinality of  $X$.

Suppose an oriented hypermap $\map{hm}$ is given by its $\sigma$-model on the set of half-edges $H=\{1,\dots, m\}$. We set $X$ to be a double of $H$, 
$X:=\{1_-,1_+,2_-,2_+,\dots,m_-,m_+\}$, and the involution $\tau_2$ to swap $i_-$ and $i_+$.
Define the permutation $\tau_0$ to be $\tau_0(i_-):=(\sigma_E(i))_+$ and 
$\tau_0(i_+):=(\sigma_E^{-1}(i))_-$. Finally, define $\tau_1$ as $\tau_1(i_-):=(\sigma_V^{-1}(i))_+$ and  $\tau_1(i_+):=(\sigma_V(i))_-$. Obviously they are involutions and the hypermap they define 
is $\map{hm}$.
\begin{figure}[!htp]
  \centering
  \includegraphics[scale=.6]{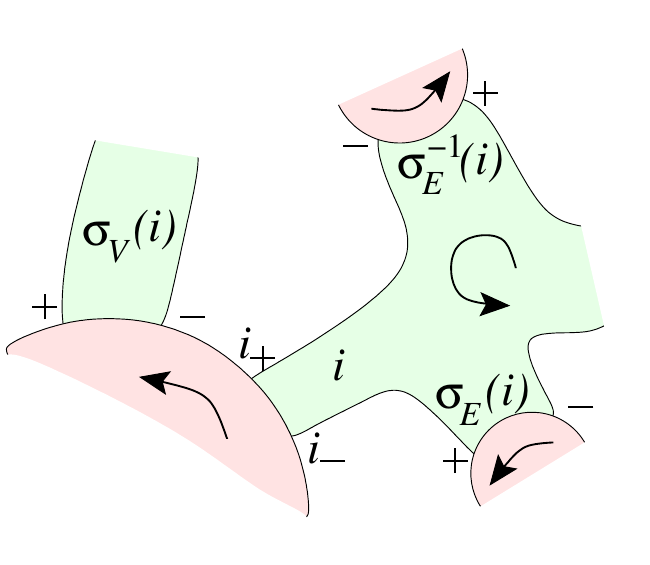}
\caption{$\sigma$ and $\tau$ permutations}
  \label{fig:sigma-flags}
\end{figure}

In the opposite way, suppose a hypermap $\map{hm}$ is given by its $\tau$-model on the set of local flags $X=\{1,\dots, n\}$. Also suppose that $\map{hm}$ is {\it connected}. That means the group generated by $\tau_0$, $\tau_1$,  and $\tau_2$ acts transitively on $X$.

For an orientable hypermap we can consistently arrange $X$ in pairs
with subscripts $+$ and $-$ as in Figure \ref{fig:sigma-flags}. For a non orientable hypermap such an arrangement is impossible. One may observe that $\tau$'s always change the subscript to the opposite one. This means that the subgroup $G$ of words of even length in $\tau$'s preserve the subscript. 
The group $G$ is generated by $\tau_2\tau_1$,  $\tau_0\tau_2$, and $\tau_1\tau_0$.
For a non orientable hypermap, the subgroup $G$ also acts transitively on $X$. In the orientable case,
$X$ splits into two orbits of $G$, one with the subscript $+$ and another one with the subscript $-$. Let $H$ be the one with subscript $+$. Then we set $\sigma_V$ (resp. 
$\sigma_E$ and $\sigma_F$) to be the restriction of $\tau_2\tau_1$  (resp. $\tau_0\tau_2$, 
and $\tau_1\tau_0$) on the orbit $H$. Obviously these restrictions satisfy the relation
$\sigma_F\sigma_E\sigma_V=(\tau_1\tau_0)(\tau_0\tau_2)(\tau_2\tau_1)=1$.
It is clear from \cref{fig:sigma-flags} that the $\sigma$-model constructed in this way gives the original orientable hypermap $\map{hm}$. The restriction to the ``$-$''-orbit gives the same hypermap with the opposite orientation.

\begin{example}\label{exa:tau-sigma}
For hypermaps on \cref{fig:NOMapEx} the subgroup $G$ is generated by the following permutations. For $\mapm_{0}$,
\begin{align*}
  \tau_2\tau_1&=(1,3,5)(2,6,4)(7,8,12)(9,11,10),\\
  \tau_0\tau_2&=(1,7)(2,10)(3,12)(4,8)(5,9)(6,11),\\ 
  \tau_1\tau_0&=(1,12)(2,11)(3,8,6,9)(4,7,5,10).
\end{align*}
For ${\mathfrak{hm}}_0$, $\tau_2\tau_1=(1,3,5)(2,6,4)$,
$\tau_0\tau_2=(1,4,3)(2,5,6)$, $\tau_1\tau_0=(1,2)(3,6)(4,510)$. In both cases the group $G$ acts transitively on flags. This is a combinatorial expression of the fact that these two hypermaps are non orientable.

On the contrary, consider the two oriented hypermaps of
\cref{fig:OMapEx}.
\begin{figure}[!htp]
  \centering
  \begin{subfigure}[c]{.4\linewidth}
    \centering
    \raisebox{-.5\height}{\includegraphics[scale=.35]{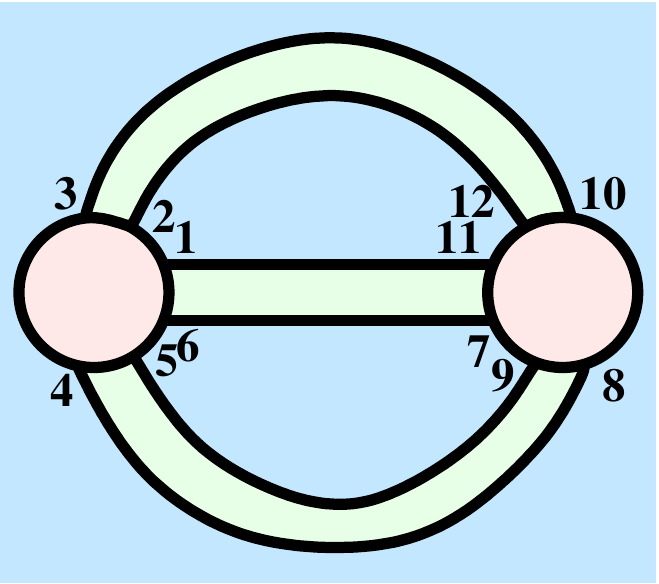}}
    \caption{An orientable map $\mapm_1$ on a sphere with two vertices and three edges}
    \label{fig-OMapEx-left}
  \end{subfigure}\qquad\qquad
  \begin{subfigure}[c]{.4\linewidth}
    \centering
    \raisebox{-.5\height}{\includegraphics[scale=.35]{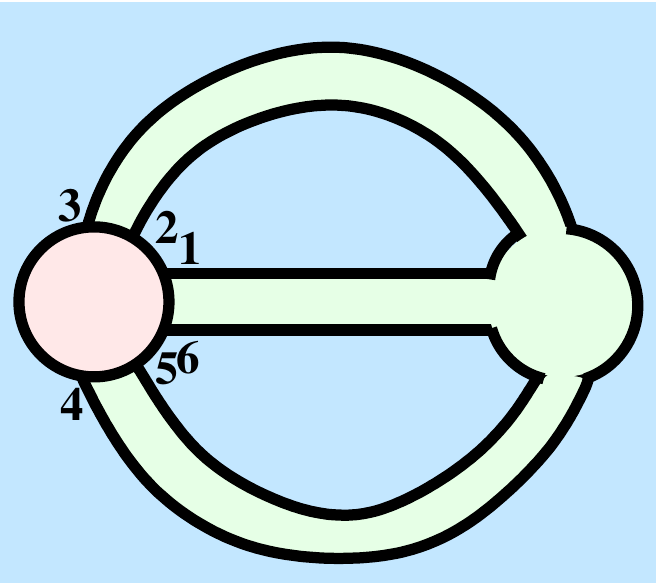}}
    \caption{An orientable hypermap ${\mathfrak{hm}}_1$ on a sphere
      with one vertex and one hyperedge}
    \label{fig-NOMapEx2-right}
  \end{subfigure}
  \caption{Map and hypermap. Second example.}
  \label{fig:OMapEx}
\end{figure}
The permutations $\tau_i$ are the following. For the map $\mapm_{1}$,
\begin{align*}
  \tau_0&=(1,11)(2,12)(3,10)(4,8)(5,9)(6,7),\\ 
  \tau_1&=(1,2)(3,4)(5,6)(7,9)(8,10)(11,12),\\
  \tau_2&=(1,6)(2,3)(4,5)(7,11)(8,9)(10,12).
\end{align*}
For ${\mathfrak{hm}}_1$, $\tau_0=(1,2)(3,4)(5,6)$,
$\tau_1=(1,2)(3,4)(5,6)$, $\tau_2=(1,6)(2,3)(4,5)$. The generators of the subgroup $G$ for $\mapm_1$ are:
\begin{align*}
  \tau_2\tau_1&=(1,3,5)(2,6,4)(7,8,12)(9,11,10),\\
  \tau_0\tau_2&=(1,7)(2,10)(3,12)(4,9)(5,8)(6,11),\\
  \tau_1\tau_0&=(1,12)(2,11)(3,8)(4,10)(5,7)(6,9).
\end{align*}
For ${\mathfrak{hm}}_1$: $\tau_2\tau_1=(1,3,5)(2,6,4)$,
$\tau_0\tau_2=(1,5,3)(2,4,6)$, $\tau_1\tau_0=1$. One can see that the group $G$ has two orbits on the set of flags. The ``$+$''-orbits are:
for $\mapm_1$, $H=\{1,3,5,7,8,12\}$; for ${\mathfrak{hm}}_1$, $H=\{1,3,5\}$. The restriction of the generators on this orbit gives the $\sigma$-models.\\
For $\mapm_1$: $\sigma_V=(1,3,5)(7,8,12)$, $\sigma_E=(1,7)(3,12)(5,8)$, 
$\sigma_F=(1,12)(3,8)(5,7)$.\\
For ${\mathfrak{hm}}_1$: $\sigma_V=(1,3,5)$, $\sigma_E=(1,5,3)$,  $\sigma_F=1$.
\end{example}

There is an elegant formula for the Euler characteristic of a hypermap in terms of its $\sigma$-model.
\begin{lemma}\label{lem-GenusCH}\cite[Proposition 1.5.3]{Lando2004aa}
Let $\mathfrak{hm}=(\sigma_V,\sigma_E,\sigma_F)$ be an 
oriented hypermap 
given by its $\sigma$ model on the set $H$ of $n$ half-edges, $n:=\#H$. 
Let $c_V$ (resp. $c_E$ and $c_F$) denote the number of cycles of $\sigma_V$ (resp. $\sigma_E$ and $\sigma_F$). Then the Euler characteristic
$\chi(\mathfrak{hm})$
of the surface of $\mathfrak{hm}$ is 
equal to
$$
    \chi(\mathfrak{hm})=c_V+c_E+c_F-n\ .
$$
\end{lemma}
\begin{proof}   Let $\mathcal T$ be the cell decomposition (tesselation) given by the hypermap $\mathfrak{hm}$. 
Note that
  \begin{itemize}
  \item $2n$ is the number of vertices of $\mathcal T$,
  \item the number of polygons in $\mathcal T$ is $c_V+c_E+c_F$,
  \item the number of edges of $\mathcal T$ is $3n$,
  \end{itemize}
Then the formula follows.
\end{proof}

\subsection{Edge Coloured Graphs}
\label{ssec-edge-coloured-graphs}

As indicated in \cref{ssec-tau-model} the boundaries of cells of a hypermap form a 3-regular graph embedded into the surface of the hypermap. It carries a natural edge colouring: the arcs of intersection of hyperedges and faces are coloured by 0, the arcs of intersection of vertices and faces are coloured by 1, and the arcs of intersection of vertices and hyperedges are coloured by 2. In this subsection we show that the entire hypermap can be reconstructed from this information.

\begin{defn} \label{def-colouredgraph}
Let $\kappa$ be a finite set. A $\kappa$-{\it\bfseries coloured graph} is
an abstract connected graph such that each edge carries a ``colour'' in $\kappa$ and each vertex is incident to exactly one edge of each colour.
\end{defn}

Note that a $\kappa$-coloured graph is necessarily
$\#\kappa$-regular and has no loops (but may contain multiple
edges). In the following, for all $I\subset\kappa$, we will denote $\kappa\setminus I$ by 
$\overline{I}$.

Let $\kappa=\{1,2,\dots,\#\kappa\}$. For a $\kappa$-coloured graph $\Gamma$ we can define a permutational $\tau$-model as a set of involutions $\tau_1,\tau_2,\dotsc,\tau_{\#\kappa}$
acting on the set $X$ of vertices of $\Gamma$ as follows.
$\tau_i$ interchange the vertices connected by an edge of colour $i$.
For the coloured graphs coming from hypermaps these permutations coincide with the $\tau$-model from \cref{ssec-tau-model}.

Each coloured graph $\Gamma$ contains some special coloured subgraphs
called {\it\bfseries bubbles} in \cite{Gurau2011aa} and {\it\bfseries residues} in
\cite{Vince1983aa}.
\begin{defn} \label{def-bubbles}
Let $\Gamma$ be a $\kappa$-coloured graph and $I\subset\kappa$. An
$\# I$-{\it\bfseries bubble} of colours $I$ in $\Gamma$ is a connected component of the $I$-coloured subgraph of $\Gamma$ induced by the edges of $\Gamma$ with colours in $I$. 
\end{defn}

In particular $0$-bubbles, corresponding to $I=\emptyset$, are the
vertices of $\Gamma$. The set of bubbles in $\Gamma$ of colours $I\subset\kappa$ is denoted by $\cB^I(\Gamma)$ or $\cB^I$ if there is no ambiguity. 
$B^I$ is its cardinality $\#\cB^I$. We also define $\cB_n(\Gamma)$,  $0\les n\les\#\kappa-1$, to be the set of all $n$-bubbles in $\Gamma$: $\cB_n:=\bigcup_{I\subset\kappa,\#I=n}\cB^I$ ; $B_n:=\#\cB_n$. Finally the whole set of bubbles of $\Gamma$, 
$\bigcup_{0\les n\les\#\kappa-1}\cB_n$, is written as $\cB(\Gamma)$.
The subgraph inclusion relation provides $\cB(\Gamma)$ with a poset structure.

\subsubsection{Topology of edge coloured graphs}
\label{sssec-topol-edge-colo}

To each coloured graph $\Gamma$, one can associate two cell complexes, $\mathbf{\Delta^{\!*}(\Gamma)}$ and its dual $\mathbf{\Delta(\Gamma)}$, as follows.

For each $D\in\N$, let $[D]$ be the set $\{0,1,\dotsc,D\}$.\\

{\bf The dual complex $\mathbf{\Delta^{\!*}(\Gamma)}$.}
\label{sec-dual-complex}
Let $\Gamma$ be a $[D]$-coloured graph. To each $D$-bubble
$b\in\cB^{\overline{i}}(\Gamma)$, one associates a $0$-simplex ${\mathfrak s}(b)$
coloured $i$. To each $(D-1)$-bubble $b\in\cB^{\overline{\{i,j\}}}$, one
associates an edge ${\mathfrak s}(b)$ the endpoints of which are respectively
coloured $i$ and $j$. In general, to each $k$-bubble
$b\in\cB^{\{i_1,\cdots,i_k\}}$, one associates an abstract $(D-k)$-simplex
${\mathfrak s}(b)$ coloured $[D]\setminus \{i_1,\cdots,i_k\}$. Now, the
poset structure of $\cB(\Gamma)$ provides gluing data for those
simplices. Indeed, let us consider two $(D-k)$-simplices 
${\mathfrak s}(b)$ and ${\mathfrak s}(b')$. If the corresponding $k$-bubbles $b$ and $b'$ are contained in a common $(k+1)$-bubble $b''$, identify ${\mathfrak s}(b)$ and ${\mathfrak s}(b')$ along their common facet ${\mathfrak s}(b'')$. This gluing respects the
colouring structure of the simplices. It can be shown that such a
complex is a trisp (for triangulated space) \cite{Kozlov2008aa}.
A.~Vince \cite[p.4]{Vince1983aa} called the topological space of this simplicial complex the {\it underlying topological space of the combinatorial map $\Gamma$}. \\

\noindent
But there is also another complex associated to $\Gamma$, dual to
$\Delta^{\!*}$.\\
{\bf The direct complex $\mathbf{\Delta(\Gamma)}$.}
\label{sec-direct-compl-delt}
It is constructed inductively, like a CW complex. To each $k$-bubble,
$0\les k\les D$, one will associate a $k$-dimensional topological
space. To each $0$-bubble $b$, i.e. to each vertex of $\Gamma$,
corresponds a point $|b|$. Each edge $e$ of $\Gamma$, i.e. each
$1$-bubble, contains two vertices $u$ and $v$. Define $|e|$ as the
cone over $|u|\cup|v|$. The realization $|e|$ of $e$ is thus a
segment. Now consider a $2$-bubble $b$. It is a bicoloured cycle in
$\Gamma$. $b$ contains a bunch of edges whose realization form a
circle. $|b|$ is defined as a cone over this circle hence a
$2$-disk. In general, let $b$ be $k$-bubble. It contains a set
$\cB_{k-1}(b)$ of $(k-1)$-bubbles. The realization of any
$b'\in\cB_{k-1}(b)$ has been defined at the previous induction
step. The realizations $|b_{1}|$ and $|b_{2}|$ for $b_{1},
b_{2}\in\cB_{k-1}$ are identified along $|b_{1}\cap b_{2}|$ (which is
a union of lower dimensional bubbles). Then the
whole set $\cB_{k-1}(b)$ has a (connected) realization that we denote
$\partial|b|$. Finally $|b|$ is defined as the cone over $\partial|b|$
(hence the name).

The realization $|\Gamma|$ of $\Gamma$ corresponds to the gluing of
the $D$-cells of $\Delta(\Gamma)$. $\Delta(\Gamma)$ is a complex whose cells are topological
spaces glued along their common boundaries. But in general, its cells
are not homeomorphic to balls. And indeed $|\Gamma|$ is generally not
a manifold but a normal pseudo-manifold \cite{Gurau2012ab}.

\subsubsection{Hypermaps as edge-coloured graphs}
\label{sssec-hyperm-edge-colo}

It was mentioned at the beginning of this subsection that a hypermap $\mathfrak{hm}$ determines a $[2]$-coloured graph $\Gamma_{\mathfrak{hm}}$. Its vertices corresponds to (local) flags 
of $\mathfrak{hm}$ and its edges of colour $i$ correspond to the orbits of the involution $\tau_i$. 

Here is an inverse construction.
Let us consider a $[2]$-coloured graph $\Gamma$. The $2$-cells of its
direct complex $\Delta(\Gamma)$ are polygons and $|\Gamma|$ is thus
the result of the gluing of polygons along common boundaries. Whereas
not true in general, the gluing of polygons, \emph{dictated by a coloured
graph}, is always a manifold and thus a closed compact (not necessarily
orientable) surface. Moreover those polygons are of three types: they
are bounded by either $01$-, $02$- or $12$-cycles ($2$-bubbles). Said
differently, $\Delta(\Gamma)$ is, in dimension $2$, a polygonal
tessellation of a closed compact (not necessarily
orientable) surface with polygons of three different types, i.e. a hypermap $\mathfrak{hm}$. Thus $[2]$-coloured graphs provide another description of hypermaps.

\begin{example}\label{exa:colgr}\rm
Here are the examples of $[2]$-coloured graph for hypermaps ${\mathfrak{hm}}_0$ from Figure \ref{fig:NOMapEx} and ${\mathfrak{hm}}_1$ from Figure \ref{fig:OMapEx}.
\begin{figure}[!htp]
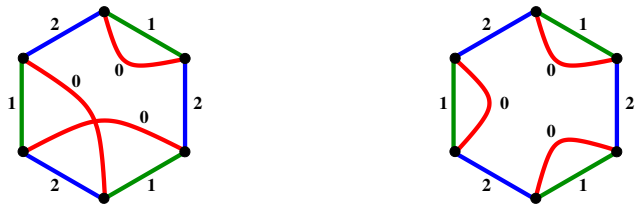

  \centering
	\risSpdf{-10}{maps-1-c2u}{}{80}{60}{0}\ \risSpdf{20}{totorl}{}{20}{0}{0}\ 
	\risSpdf{-10}{cgr-c2}{}{70}{0}{0}\hspace{1cm}
	\risSpdf{-10}{maps-1-c4u}{}{80}{0}{0}\ \risSpdf{20}{totorl}{}{20}{0}{0}\ 
 	\risSpdf{-10}{cgr-c4}{}{70}{0}{0}\hspace{3cm}
 \caption{$[2]$-coloured graphs $\Gamma_{{\mathfrak{hm}}_0}$ and $\Gamma_{{\mathfrak{hm}}_1}$} 
  \label{fig:colgr}
\end{figure}
A reader may enjoy constructing the direct complexes 
$\Delta(\Gamma_{{\mathfrak{hm}}_0})$ and
$\Delta(\Gamma_{{\mathfrak{hm}}_1})$, and checking that they are indeed  isomorphic to the hypermaps from Figures \ref{fig:NOMapEx} and
\ref{fig:OMapEx}.
\end{example}

\begin{lemma}\label{lemma:or-cgr}
A hypermap corresponding to a $[2]$-coloured graph $\Gamma$ is orientable if and only if $\Gamma$ is bipartite.
\end{lemma}
\begin{proof} According to \cref{ssec-sigma-model} a hypermap is orientable if and only if the vertices of $\Gamma$ can be split into two parts with subscripts $+$ and $-$ as on \cref{fig:sigma-flags}. 
\end{proof}

\begin{rem}
In \cite{Vince1983aa} A.~Vince proposed a way to associate a $[d]$-coloured graph $\Gamma$ to any cell decomposition $K$ of a closed $d$-manifold. $\Gamma$ is defined as the $1$-skeleton of the complex dual to the first barycentric subdivision of $K$. Whereas his method works for any cell complex associated to closed manifolds, it does not define a one-to-one correspondence between hypermaps and $[2]$-coloured graphs (not all coloured graphs have a dual complex which is the barycentric subdivision of another cell complex). Moreover the coloured graph thus associated to $K$ is of higher order than ours.
\end{rem}

\section{Partial duality}
\label{sec-pdual}

\subsection{Definition}
\label{ssec-definition}

Assume that a hypermap $\mathfrak{hm}$ is connected. Otherwise we will need to do partial duality for each connected component separately and then take the disjoint union.
Let $S$ be a subset of cells of $\mathfrak{hm}$ of the same type, either vertex-cells, or hyperedge-cells, or face-cells. We will define the {\it partial dual hypermap $\mathfrak{hm}^S$} relative to $S$. If $S$ is the set of all cells of the given type, the partial duality relative to $S$ is the total duality which swaps the two types of the remaining cells without changing the cells themselves and reverses the orientation in an oriented case. 

For example, if $\mathfrak{hm}$ is a graph cellularly embedded into a
surface then the total duality relative to the whole set of edges is
the classical duality of graphs on surfaces which interchanges
vertices and faces. Since the concept of hypermap is completely
symmetrical we can make the total duality relative to the set of
vertices for example. Then the edges and faces will be
interchanged. The hypermap ${\mathfrak{hm}}_1$ from \cref{fig:OMapEx}
has one vertex, one hyperedge and 3 faces. So we have three total
duals relative to the vertex, relative to the hyperedge, and relative
to all three faces, which differ only by the colour (type) of the
corresponding cells. On \cref{fig:tduals} the three duals are shown as cell decomposed spheres together with the corresponding embeddings of the hypergraphs; the hyperedges are embedded as one-dimensional stars centered at little squares.
\begin{figure}[!htp]
  \centering\thicklines
	\risSpdf{-10}{maps-1-c4u}{\put(0,-5){\vector(-3,-2){50}}
	    \put(45,-5){\vector(0,-1){30}}\put(90,-5){\vector(3,-2){50}}}{90}{70}{50}\\
  \includegraphics[scale=.3]{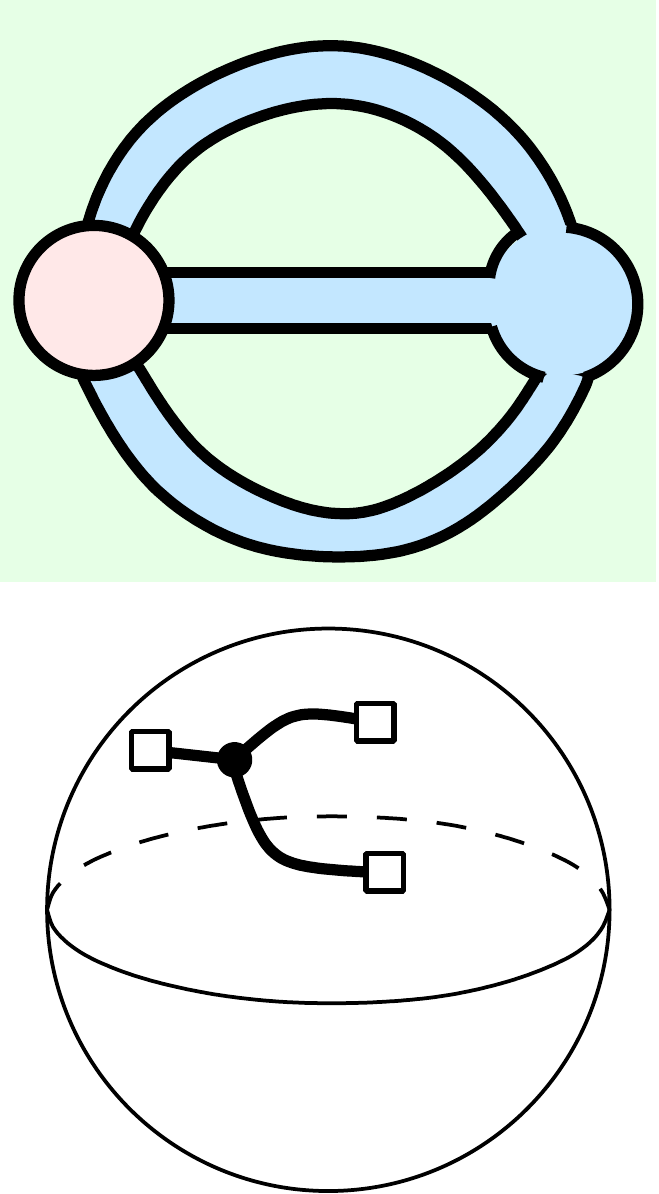}\qquad
  \includegraphics[scale=.3]{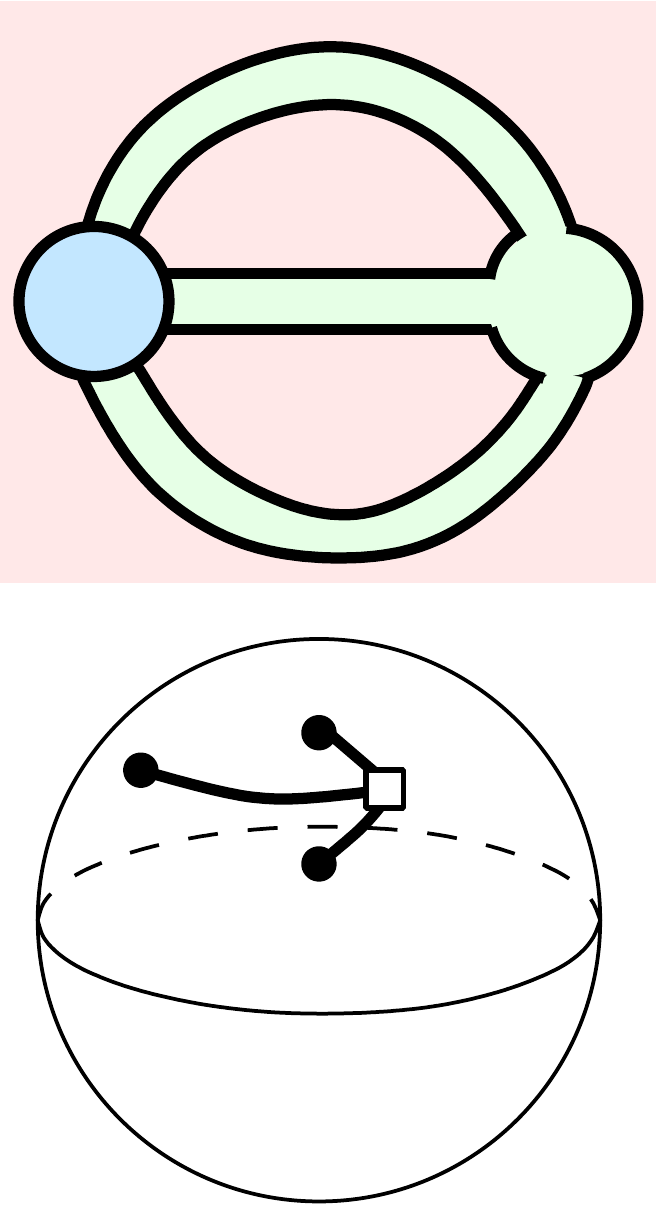}\qquad
  \includegraphics[scale=.3]{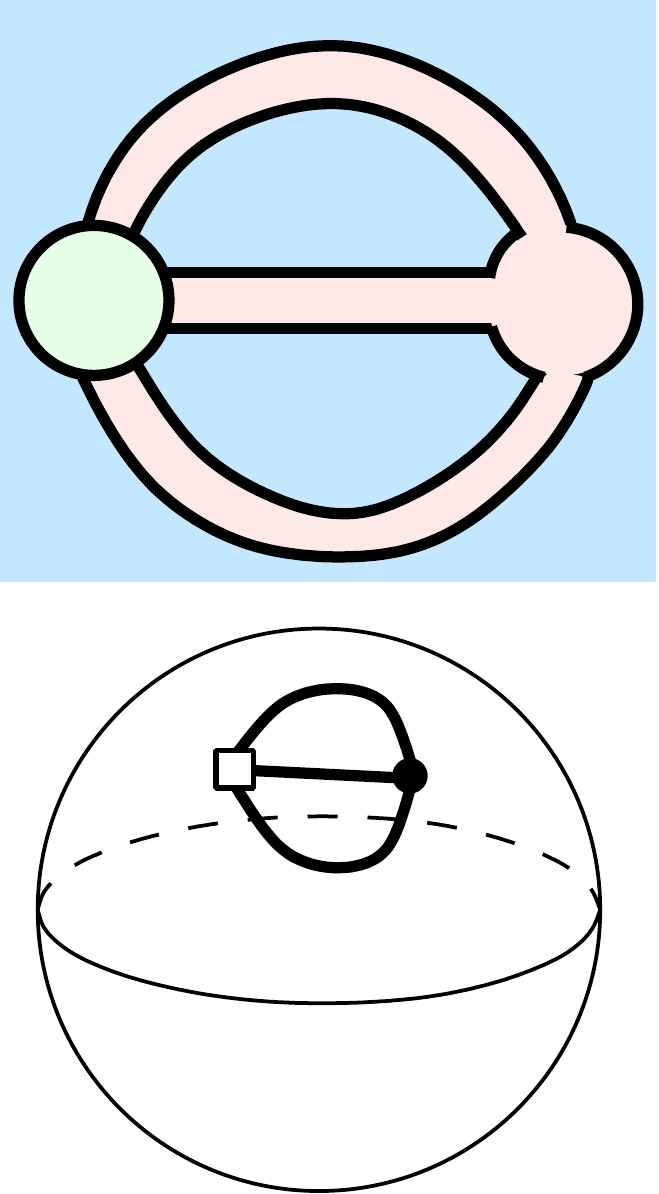}
  \caption{Total duals of the hypermap ${\mathfrak{hm}}_1$ from \cref{fig:OMapEx}}
  \label{fig:tduals}
\end{figure}
The left picture represents ${\mathfrak{hm}}_1^{\{v\}}$ and has 3 hyperedges with a single half-edge each. The middle picture represents ${\mathfrak{hm}}_1^{\{e\}}$ with a single hyperedge of valency 3 adjacent to 3 distinct vertices and a single face. The right picture ${\mathfrak{hm}}_1^{\{f_1,f_2.f_3\}}$ is isomorphic to the original hypermap ${\mathfrak{hm}}_1$.

\begin{defn}\label{def:pdual}
Without loss of generality we may assume that $S$ is a subset of the
set of vertex-cells.  Choose a different type of cells, say hyperedges. Later in 
\cref{lemma:ind-type} we show that the resulting hypermap does not depend on this choice; we could choose faces instead of hyperegdes if we want to.

{\bf Step 1.} Consider the boundary $\partial F$ of the surface $F$ which is the union of  the cells from $S$ and all cells of the chosen type, hyperedges in our case.

{\bf Step 2.} Glue a disk to each connected component of 
$\partial F$. These will be the {\it\bfseries hyperedge-cells} for 
$\mathfrak{hm}^S$. Not that we do not include the interior of $F$ into the hyperedges. Although if $\partial F$ has only one component, gluing a disk to it results in the surface $F$ itself, and then we may consider $F$ as the single hyperedge of $\mathfrak{hm}^S$. See \cref{fig:pd-st-1-2}.

{\bf Step 3.} Take a copy of every vertex. These disks will be the
{\it\bfseries vertex-cells} for  $\mathfrak{hm}^S$. Their attachment to the
hyperedges is as follows. Every vertex disk of the original hypermap $\mathfrak{hm}$ contributes one or several intervals to $\partial F$. Indeed, if a vertex belong to $S$, then it contributes to $F$ itself and
a part of its boundary contributes to $\partial F$. If a vertex is not
in $S$, then it has some hyperededges attached to it because
$\mathfrak{hm}$ is assumed to be connected. So such a vertex-disk has
a common boundary intervals with $F$ and therefore contributes these
intervals to $\partial F$. The new copies of the vertex-disks, as
vertices of $\mathfrak{hm}^S$, are attached to hyperedges exactly
along the same intervals as the old ones. See \cref{fig:pd-st-3}.

{\bf Step 4.} At the previous steps we constructed the vertex and
hyperedge cells for the partial dual $\mathfrak{hm}^S$. Their union
forms a surface with boundary. Glue a disk to each of its boundary
components. These are going to be the {\it\bfseries faces} of $\mathfrak{hm}^S$.
See \cref{fig:pd-st-4}.

This finishes the construction of the partial dual hypermap $\mathfrak{hm}^S$.
\end{defn}

\begin{example}\label{exa:pdual}\rm
We examplify the construction of the partial dual $\mapm_1^{\{v\}}$
for the map $\mapm_1$ from Figure \ref{fig:OMapEx} relative to its left vertex $v$.
\begin{figure}[!htp]
  \centering
  \rb{32pt}{$\mapm_1\ =\ $}\includegraphics[scale=.3]{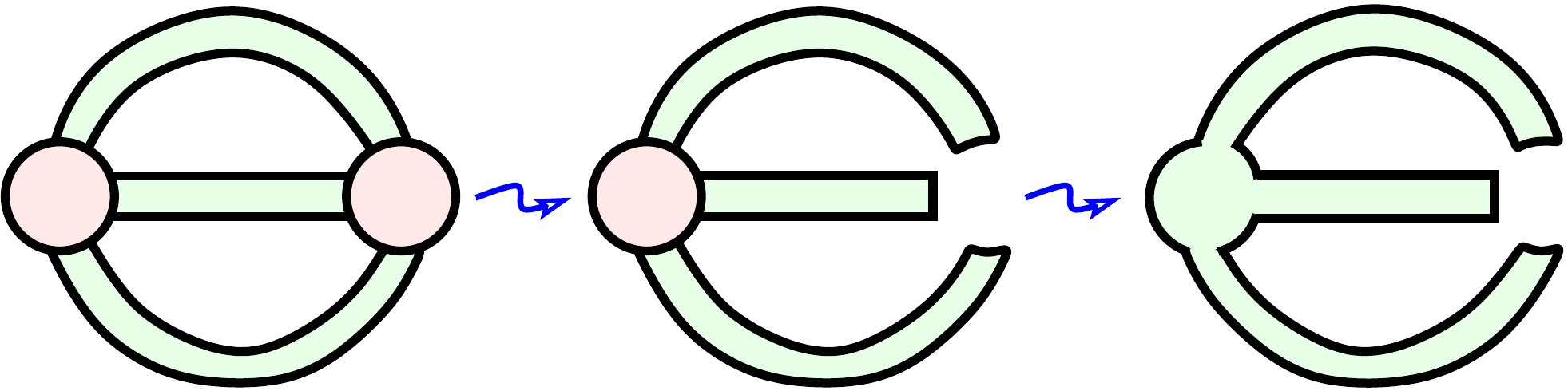}
  \caption{Steps 1 \& 2: forming hyperedges of $\mapm_1^{\{v\}}$}
  \label{fig:pd-st-1-2}
\end{figure}
\begin{figure}[!htp]
  \centering
  \includegraphics[scale=.3]{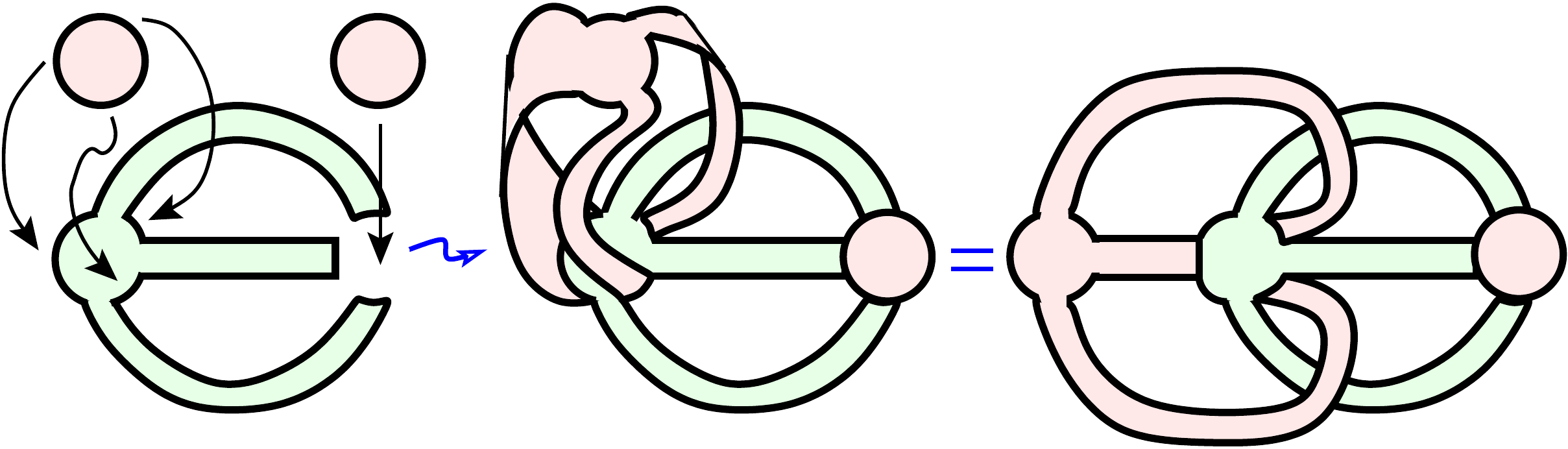}
  \caption{Step 3: copying vertices and gluing them to hyperedges}
  \label{fig:pd-st-3}
\end{figure}
\begin{figure}[!htp]
  \centering
  \includegraphics[scale=.3]{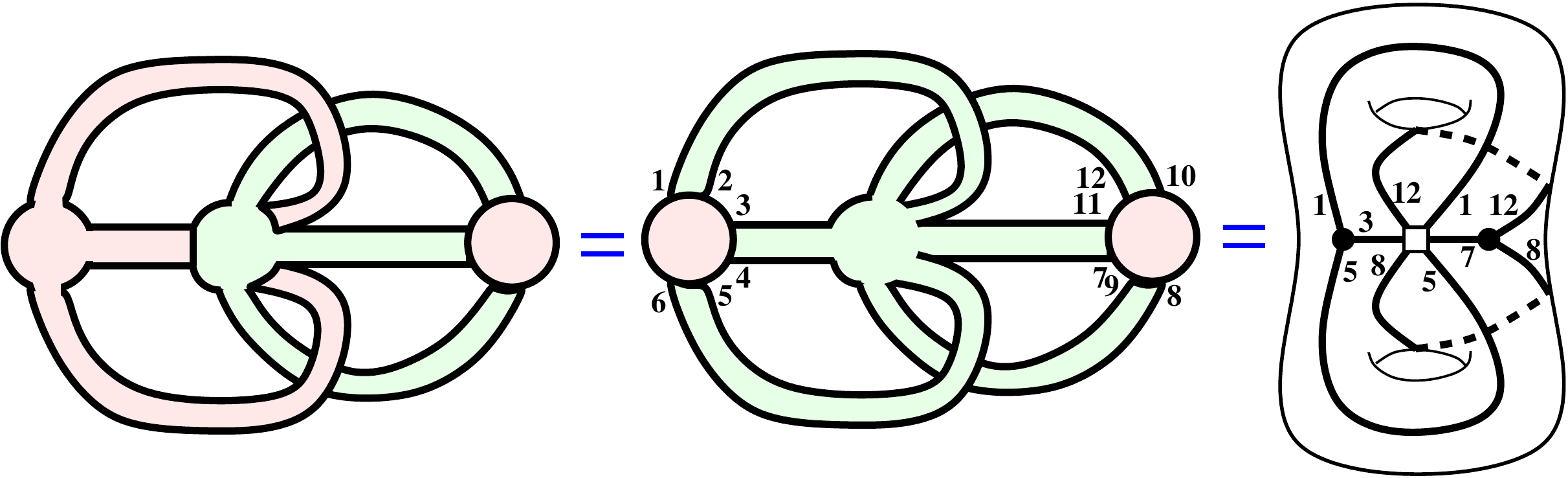}
  \caption{Step 4: gluing faces and the resulting hypermap $\mapm_1^{\{v\}}$.}
  \label{fig:pd-st-4}
\end{figure}
\end{example}

Similarly one may find the partial dual $\mapm_0^{\{v\}}$ for the non orientable map $\mapm_0$ from \cref{fig:NOMapEx}. The resulting surface after step 3 will be similar to the one above, only one half-edge will be twisted. It still has one boundary component, and therefore a single face. So its Euler characteristic is still $-2$, only now the resulting hypermap will be non orientable. It represents a surface homeomorphic to a connected sum of 4 copies of the projective plane.

\begin{lemma}\label{lemma:ind-type}
The resulting hypermap does not depend on the choice of type at the beginning of \cref{def:pdual}.
\end{lemma}
\begin{proof}
Decompose the boundary circles of faces on a hypermap $\mathfrak{hm}$ into the union of three sets of arcs intersecting only at the end points of the arcs, 
$D_0(\mathfrak{hm})\cup D_{1,S}(\mathfrak{hm})\cup D_{1.\not{S}}(\mathfrak{hm})$. The set $D_0(\mathfrak{hm})$ consist of arc of intersection of faces with hyperedges, $D_{1,S}(\mathfrak{hm})$  --- of faces with vertices from the set $S$, and $D_{1.\not{S}}(\mathfrak{hm})$ --- of faces with vertices not from $S$.
Analyzing the result of Step 3 of the construction one can easily note that $D_0(\mathfrak{hm})=D_0(\mathfrak{hm}^S)$ and
$D_{1.\not{S}}(\mathfrak{hm})=D_{1.\not{S}}(\mathfrak{hm}^S)$.
Moreover, $D_{1,S}(\mathfrak{hm}^S)$ consists of the complementary
arcs of the boundary circles of vertices from $S$ to the arcs
$D_{1,S}(\mathfrak{hm})$; formally the complementary arcs on the
second copies of the vertices of $S$. This means that the boundary
circles of faces of $\mathfrak{hm}^S$ are exactly the boundary circles
of the surface obtained by the union of vertices of $S$ and all the
faces. In other words on Step 1 we may take faces instead of hyperedges and we will get the same boundary circles as for $\mathfrak{hm}^S$. Then, by symmetry the hyperedges will also be the same.
\end{proof}

Analogously to \cite[Sec.1.8]{Chmutov2007aa} the following lemma describes simple properties of the partial duality for hypermaps. Its proof is obvious.

\begin{lemmab}\label{lemma:prop}
\begin{itemize}
\item[(a)] $\del[0]{\mathfrak{hm}^S}^S= \mathfrak{hm}$.
\item[(b)] There is a bijection between the cells of type $S$ in
$\mathfrak{hm}$ and the cells of the same type in $\mathfrak{hm}^S$. 
This bijection preserves the valency of cells.
The number of cells of other types may change.
\item[(c)] If $s\not\in S$ but has the same type as the cells of $S$, then 
$\mathfrak{hm}^{S\cup\{s\}} = \bigl(\mathfrak{hm}^S\bigr)^{\{s\}}$.
\item[(d)] $\bigl(\mathfrak{hm}^{S'}\bigr)^{S''}= \mathfrak{hm}^{\Delta(S',S'')}$, where 
    $\Delta(S',S''):=(S'\cup S'')\setminus (S'\cap S'')$ is the symmetric difference of sets.
\item[(e)] Partial duality preserves orientability of hypermaps.
\end{itemize}
\end{lemmab}

\subsection{Partial duality in $\mathbf{\sigma}$-model}
\label{ssec-pdual-s}

For an oriented hypermap $\mathfrak{hm}$ represented in the $\sigma$-model of \cref{ssec-sigma-model} we shall write
$$\mathfrak{hm} = (\sigma_V, \sigma_E, \sigma_F)\ .
$$

\begin{thm}\label{thm:pd-s}
Let $S$ be a subset $S:=V'$ of vertices (resp. subset of hyperedges $S:=E'$ and subset of faces $S:=F'$) of a hypermap $\mathfrak{hm}$. Then its partial dual is given by the permutations
$$\begin{array}{ccl}
\mathfrak{hm}^{V'}&=&(\sigma_{\overline{V'}}\sigma_{V'}^{-1},\ 
           \sigma_E\sigma_{V'},\ \sigma_{V'}\sigma_F) \vspace{10pt}\\
\mathfrak{hm}^{E'}&=&(\sigma_{E'}\sigma_V,\        
        \sigma_{\overline{E'}}\sigma_{E'}^{-1},\ \sigma_F\sigma_{E'})
        \vspace{10pt}\\
\mathfrak{hm}^{F'}&=&(\sigma_V\sigma_{F'},\ \sigma_{F'}\sigma_E,\    
        \sigma_{\overline{F'}}\sigma_{F'}^{-1})\ ,
\end{array}
$$
where $\sigma_{V'}$, $\sigma_{E'}$, $\sigma_{F'}$ denote the
permutations consisting of the cycles corresponding to the elements of $V'$, $E'$, $F'$ respectively, and overline means the complementary set of cycles. 
\end{thm}

Similar formulas in the particular case of maps were announced in \cite{Gross2020ab} 
(see also \cite[Section 5.2]{Gross2020aa}).

\begin{proof}
Because of the symmetry it is sufficient to prove the theorem in the
case $S=V'$. From \cref{def:pdual} it follows that the number of half-edges is preserved by the partial duality. We need to make a bijection between half-edges of $\mathfrak{hm}$ and those of $\mathfrak{hm}^S$ such that the corresponding permutations are related as in the first equation of the theorem. 

Half-edges are attached to vertices. If a vertex does not belong to
$S$ then the attachment of half-edges to it does not change with
partial duality (Step 2). So for those half-edges the required
bijection is the identity.

Half-edges attached to vertices of $S$ change, so we need to indicate the bijection for them.
Consider a vertex-disk from the set $S$ of the original hypermap $\mathfrak{hm}$. It can be represented as a
$2k$-gon because the arcs of its boundary circle intersecting with hyperedges and faces alternate. In $\mathfrak{hm}$ it has $k$ half-edges
attached along every other side. We call them {\it old half-edges}. These half-edges together with the vertex-disk form a piece of the surface $F$ on Step 1 near the vertex. The orientation of $F$ induces an orientation on its boundary $\partial F$. In the partial dual 
$\mathfrak{hm}^S$ the hyper-edges, the {\it new hyperedges}, are
attached to every connected component of $\partial F$ (Step 2). The
orientation of $\partial F$ induces the orientation on new
hyperedges. They are attached to a new vertex (Step 3) along the other
sides of the $2k$-gon, which form {\it new half-edges}. Set the label
of a new half-edge to be the same as the label of the old one preceding the new half-edge in the direction of the orientation of the old vertex. This gives the bijection of half-edges around vertices of $S$.
The orientation on the new vertex, as well as on the entire hypermap $\mathfrak{hm}^S$, is induced from the new hyperedges.
\begin{figure}[!htp]
  \centering
  \includegraphics[scale=.5]{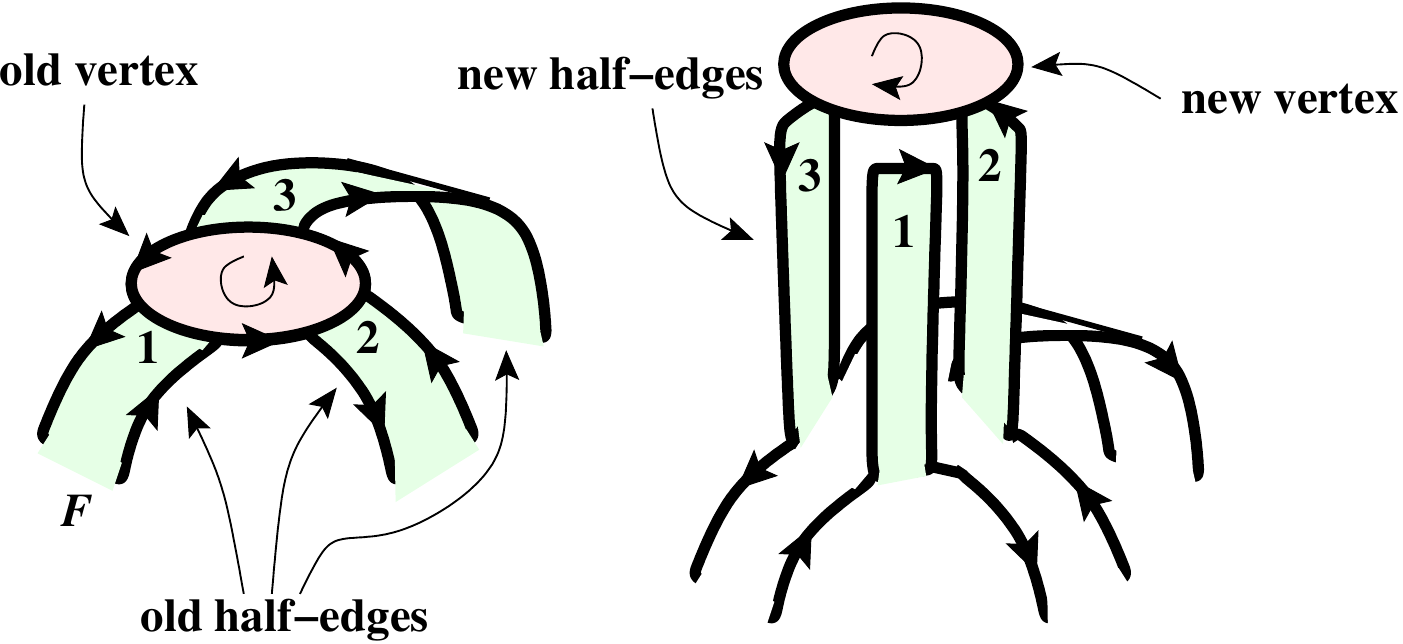}
  \caption{Permutation $\sigma_V(\mathfrak{hm}^S)$}
  \label{fig:pd-new-vert}
\end{figure}
Figure \ref{fig:pd-new-vert} shows that the labels of the new
half-edges appear around new vertices in the order opposite to the one
around old ones. This means that the cycle in the permutation $\sigma_V$ corresponding to a vertex in $S$ of the initial hypermap $\mathfrak{hm}$ should be inverted
to get the cycle for $\mathfrak{hm}^S$. This proves
that the first term of the first formula of the theorem
$\sigma_V(\mathfrak{hm}^S)= \sigma_{V'}^{-1}(\mathfrak{hm}) \sigma_{\overline{V'}}(\mathfrak{hm})$.

For the second term we need to analyze the cyclic order of new
half-edges around new hyperedge according to its orientation. It may
be read off from labels of the half-edges met when traveling along the boundary of the hyperedge in the direction of its orientation. Such a boundary for $\mathfrak{hm}^S$ is exactly a connected component of  $\partial F$ with the orientation induced from $\mathfrak{hm}$. The last is given precisely by the product of permutations 
$\sigma_E(\mathfrak{hm})\sigma_{V'}(\mathfrak{hm})$. 
Indeed, consider Figure \ref{fig:pd-he}  and suppose that 
$\sigma_E(\mathfrak{hm}): 2\mapsto i$ for some $i$. Then the new half-edge labels appear at $\partial F$ in the order $\dots, 1,i,\dots$. So $\sigma_E(\mathfrak{hm}^S): 1\mapsto i$, which is equal to
$\sigma_E(\mathfrak{hm})\sigma_{V'}(\mathfrak{hm})$:
$$\xymatrixcolsep{5pc}\xymatrix{
1 \ar@{|->}[r]^{\sigma_{V'}(\mathfrak{hm})} &2
\ar@{|->}[r]^{\sigma_E(\mathfrak{hm})} &i}\ .
$$
\begin{figure}[!htp]
  \centering
  \includegraphics[scale=.5]{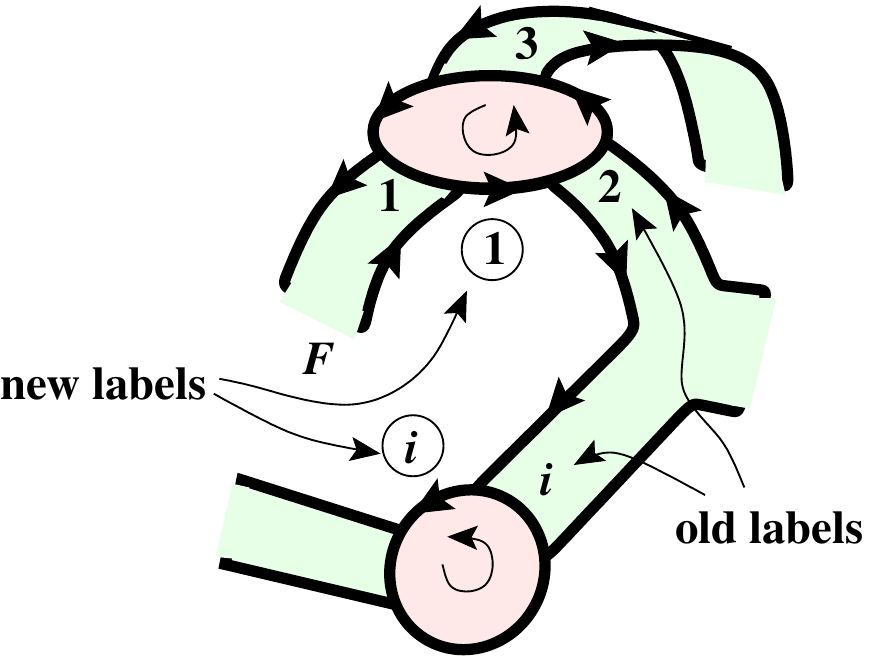}
  \caption{Permutation $\sigma_E(\mathfrak{hm}^S)$}
  \label{fig:pd-he}
\end{figure}
This proves the second term.

The third term follows from the relation $\sigma_F\sigma_E\sigma_V=1$.
\end{proof}

\begin{example}\label{exa:tau-sigma-1}
This is a continuation of \cref{exa:tau-sigma}.
We found that for the map $\mapm_1$ on \cref{fig:OMapEx} the permutations $\sigma$'s act on the set of half-edges $H=\{1,3,5,7,8,12\}$ as 
$$\sigma_V=(1,3,5)(7,8,12),\quad \sigma_E=(1,7)(3,12)(5,8), \quad 
\sigma_F=(1,12)(3,8)(5,7).$$
The cycle $(1,3,5)$ of $\sigma_V$ corresponds to the the left vertex $v$.
\begin{figure}[!htp]
  \centering
\rb{32pt}{$\mapm_1\ =\ $}\risSpdf{-10}{maps-1-c3u}{\put(11,42){$v$}}{100}{80}{0}
\quad \risSpdf{30}{totor}{}{30}{0}{0}\quad 
\rb{32pt}{$\mapm_1^{\{v\}}\ =\ $}\risSpdf{-10}{pdual-4sm}{}{130}{0}{0}
  \caption{Partial duality in $\sigma$-model.}
  \label{fig:Ex-sigma-m}
\end{figure}

For the $\sigma$-model of the partial dual $\mapm_1^{\{v\}}$, we set $V'=\{v\}$. Then $\sigma_{V'}=(1,3,5)$ and 
$\sigma_{\overline{V'}}=(7,8,12)$. According to the theorem
$$\begin{array}{rcl}
\sigma_V(\mapm_1^{\{v\}}) &=& \sigma_{\overline{V'}}\sigma_{V'}^{-1}=
     (1,5,3)(7,8,12),\\
\sigma_E(\mapm_1^{\{v\}}) &=& \sigma_E\sigma_{V'}=
     (1,7)(3,12)(5,8)(1,3,5)=(1,12,3,8,5,7),\\
\sigma_F(\mapm_1^{\{v\}}) &=& \sigma_{V'}\sigma_F = 
     (1,3,5)(1,12)(3,8)(5,7)=(1,12,3,8,5,7).
\end{array}
$$
One may check that these permutations agree with the last picture on \cref{fig:pd-st-4,fig:Ex-sigma-m}.
\end{example}

\begin{cor}\label{cor:tds}
The total duality with respect to $S:=V$ (resp. $S:=E$ and $S:=F$) is reduced to the classical Euler-Poincar\'e duality which swaps the names of two remaining types of cells and reverse the orientation.

In $\sigma$-model it is given by the formulae
$$\begin{array}{ccl}
\mathfrak{hm}^V&=& (\sigma_V^{-1},\ \sigma_E\sigma_V,\ \sigma_V\sigma_F) 
   =(\sigma_V^{-1},\ \sigma_F^{-1},\ \sigma_E^{-1})\vspace{10pt}\\
\mathfrak{hm}^E&=& (\sigma_E\sigma_V,\ \sigma_E^{-1},\ \sigma_F\sigma_E) 
   =(\sigma_F^{-1},\ \sigma_E^{-1},\ \sigma_V^{-1})\vspace{10pt}\\
\mathfrak{hm}^F&=& (\sigma_V\sigma_F,\ \sigma_F\sigma_E,\ \sigma_F^{-1}) 
   =(\sigma_E^{-1},\ \sigma_V^{-1},\ \sigma_F^{-1})\ .
\end{array}
$$
\end{cor}
The inverse of these permutations are responsible for the change of orientation of the hypermap.

\subsection{Partial duality in $\mathbf{\tau}$-model}
\label{ssec-pdual-t}

\begin{thm}\label{thm:pd-t}
Consider the $\tau$-model of a hypermap $\mathfrak{hm}$ given by the permutations
$$\tau_0(\mathfrak{hm}):(v,e,f)\mapsto (v',e,f), \quad
\tau_1(\mathfrak{hm}):(v,e,f)\mapsto (v,e',f), \quad
\tau_2(\mathfrak{hm}):(v,e,f)\mapsto (v,e,f')
$$ of its local flags. 
Let $V'$ be a subset of its vertices, $\tau_1^{V'}$ be the product of all transpositions in $\tau_1$ for $v\in V'$, and $\tau_2^{V'}$ be the product of all transpositions in $\tau_2$ for $v\in V'$.
Then its partial dual $\mathfrak{hm}^{V'}$ is given by the permutations
$$\tau_0(\mathfrak{hm}^{V'})=\tau_0,\qquad 
\tau_1(\mathfrak{hm}^{V'})=\tau_1 \tau_1^{V'} \tau_2^{V'},\qquad  
\tau_2(\mathfrak{hm}^{V'})=\tau_2 \tau_1^{V'} \tau_2^{V'}\ .
$$
In other words the permutations $\tau_1$ and $\tau_2$ swap their transpositions of local flags around the vertices in $V'$.
Similar statements hold for partial dualities relative to the subsets of hyperedges $E'$ and of faces $F'$.
\end{thm}

A particular case of these formulas for maps
was rediscovered in
\cite[Equation 14]{Gross2020ac} and announced in \cite{Gross2020ab} 
(see also \cite[Section 5.2]{Gross2020aa}).

\begin{proof} From \cref{def:pdual} one may see that if a vertex does  not participate in the partial duality, $v\not\in V'$, then nothing changes with local flags around it. But if $v\in V'$, then the roles of edges and faces in its local flags are interchanged. This may be seen on Step 3 and also on \cref{fig:pd-new-vert,fig:pd-he} when the second copy of the vertex is attached to the new hyperedges. So, if two such local flags were transposed by $\tau_1$ of the original hypermap, then they will be transposed by $\tau_2$ of the partial dual and vise versa.
\end{proof}

\begin{example}\label{exa:tau-sigma-2}
In \cref{exa:tau-sigma}, we found the $\tau$-model  for the map
$\mapm_1$ of \cref{fig:OMapEx}:
\begin{align*}
\tau_0&=(1,11)(2,12)(3,10)(4,8)(5,9)(6,7),\quad 
\tau_1=(1,2)(3,4)(5,6)(7,9)(8,10)(11,12),\\
\tau_2&=(1,6)(2,3)(4,5)(7,11)(8,9)(10,12).
\end{align*}
\begin{figure}[!htp]
  \centering
\rb{32pt}{$\mapm_1\ =\ $}\risSpdf{-10}{maps-1-c3}{\put(11,42){$v$}}{100}{80}{0}
\quad \risSpdf{30}{totor}{}{30}{0}{0}\quad 
\rb{32pt}{$\mapm_1^{\{v\}}\ =\ $}\risSpdf{-10}{pdual-4tm}{}{130}{0}{0}
  \caption{Partial duality in $\tau$-model.}
  \label{fig:Ex-tau-m}
\end{figure}
There are six flags around the left vertex $v$ labeled by $1,2,\dotsc,
6$. The corresponding transpositions around this vertex are
$$\tau_1^{\{v\}}=(1,2)(3,4)(5,6),\qquad 
\tau_2^{\{v\}}=(1,6)(2,3)(4,5).
$$
Swapping them between $\tau_1$ and $\tau_2$, we get the $\tau$-model of the partial dual
\begin{align*}
\tau_0(\mathfrak{hm}^{\{v\}})&=(1,11)(2,12)(3,10)(4,8)(5,9)(6,7),\\
\tau_1(\mathfrak{hm}^{\{v\}})&=(1,6)(2,3)(4,5)(7,9)(8,10)(11,12), \\
\tau_2(\mathfrak{hm}^{\{v\}})&=(1,2)(3,4)(5,6)(7,11)(8,9)(10,12).
\end{align*}
This agrees with the labeling of flags on \cref{fig:pd-st-4,fig:Ex-tau-m}.
\end{example}

\begin{cor}\label{cor:tdt}
The $\tau$-model of the total dual $\mathfrak{hm}^{V}$ of a hypermap $\mathfrak{hm}$ is given by the involutions
$$\tau_0(\mathfrak{hm}^{V}) = \tau_0(\mathfrak{hm}),\qquad
\tau_1(\mathfrak{hm}^{V}) = \tau_2(\mathfrak{hm}),\qquad
\tau_2(\mathfrak{hm}^{V}) = \tau_1(\mathfrak{hm}).
$$
\end{cor}
One may check that this agrees with \cref{cor:tds} in the case of oriented hypermaps.

\subsection{Partial duality for coloured graphs}
\label{ssec-pdual-cgr}

Let $\Gamma_{\mathfrak{hm}}$ be the $[2]$-coloured graph corresponding to a hypermap $\mathfrak{hm}$. Let $I$ be a subset of two out of three colours, for example 
$I=\{1,2\}$, and let $S$ be a subset of 2-bubbles in
$\cB^I$ which corresponds to a subset of vertices of $\mathfrak{hm}$.

\begin{thm}\label{thm:pd-cg}
The $[2]$-coloured graph $\Gamma_{\mathfrak{hm}^S}$ of the partial dual hypermap $\mathfrak{hm}^S$ is obtained from $\Gamma_{\mathfrak{hm}}$
by swapping the colours 1 and 2 for all edges in the 2-bubbles of $S$.
In particular, the underlying graphs of $\Gamma_{\mathfrak{hm}^S}$ and 
$\Gamma_{\mathfrak{hm}}$ are the same.
\end{thm}
Ellingham and Zha \cite{Ellingham2015aa} obtained a similar result in the case of maps.

\begin{proof} The edges of $\Gamma_{\mathfrak{hm}}$ of colour 1 (resp. 2) correspond to 2-element orbits of $\tau_1$ (resp. $\tau_2$).
According to \cref{thm:pd-t} the partial dual hypermap is obtained by swapping the corresponding transpositions of $\tau_1$ and $\tau_2$. This corresponds to swapping the colours  1 and 2 in the bubbles of $S$.
\end{proof}

\begin{example}\label{exa:td-cgr}\rm
Here are the $[2]$-coloured graphs $\Gamma_{\mathfrak{hm}_1^V}$, $\Gamma_{\mathfrak{hm}_1^E}$ and $\Gamma_{\mathfrak{hm}_1^F}$ for the total duals of the hypermap $\mathfrak{hm}_1$ from \cref{fig:OMapEx} relative to the set of all vertices  $V$, all edges $E$, and all faces $F$. These dual hypermaps are shown on \cref{fig:tduals}.
\begin{figure}[!htp]
  \centering
  \includegraphics[scale=.3]{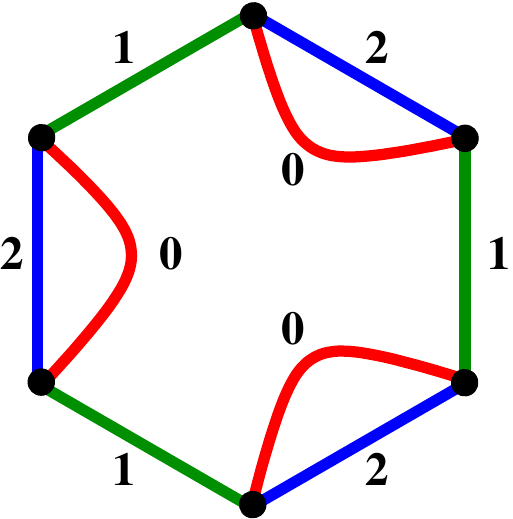}\qquad
  \includegraphics[scale=.3]{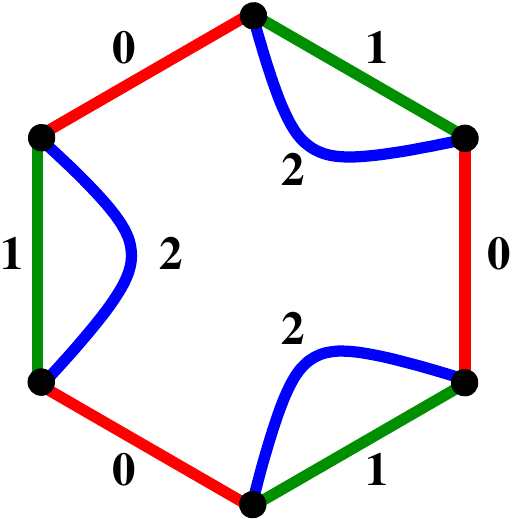}\qquad
  \includegraphics[scale=.3]{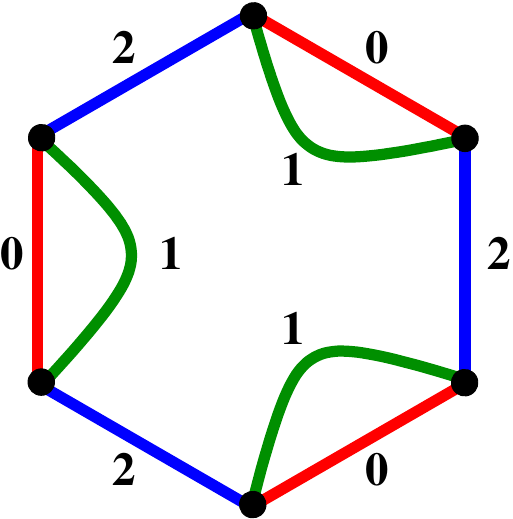}
  \caption{$[2]$-coloured graphs of total duals $\Gamma_{\mathfrak{hm}_1^V}$, $\Gamma_{\mathfrak{hm}_1^E}$, $\Gamma_{\mathfrak{hm}_1^F}$}
  \label{fig:td-cgr}
\end{figure}
\end{example}

\begin{rem}[Higher dimensional partial duality]
Such an easy interpretation of partial duality for $[2]$-coloured graphs
easily allows to make a higher dimensional generalization for $[D]$-coloured graphs $\Gamma$. Namely, fix a set $I$ of $D$ colours out of the total number of $D+1$ colours, and let $S$ be a subset of $D$-bubbles in $\cB^I$. The {\it partial dual $\Gamma^S$ relative to 
$S$} is a $[D]$-coloured graph obtained from $\Gamma$ by a permutation of
the colours of the edges in $S$.

In this case the word ``duality'' is inappropriate. It is rather an
action of a symmetric group $S_D$ on colours of edges of bubbles of
$S$. In the hypermap case, $D=2$. This group is isomorphic to $\Z_2$, so the partial duality corresponds to the only nontrivial element of order 2. But for higher $D$ the group $S_D$ contains higher order elements so they will not be ``dualities'' anymore.

This concept of higher dimensional partial duality is completely unexplored up to now. It would be very interesting to study it. In particular, is it true that if the realization $|\Gamma|$ of $\Gamma$ through its direct complex $\Delta(\Gamma)$ is a manifold, then the realization of its partial dual $|\Gamma^S|$ is also a manifold? How partial duality affects the (co)homology groups $H_*(\Delta(\Gamma))$?
\end{rem}

\section{Genus change}
\label{sec-genus}

The {\it\bfseries Euler genus} $\gamma$ is equal to twice the genus for orientable hypermaps and to the number of M\"obius bands $\mu$ in presentation of the surface of hypermaps as spheres with $\mu$ bands in them in the unorientable case. The bijection between hypermaps and $[2]$-coloured graphs, see \cref{sssec-hyperm-edge-colo}, allows to
derive a simple formula for the Euler genus change under partial duality, in terms of change of the numbers of bicoloured cycles (or
$2$-bubbles). In the case of maps, it expresses the genus change in
terms of certain induced subgraphs of the map and of its total dual.
\begin{defn}[Special subgraphs]
  \label{def-SpecialSubgraphs}
  Let $\Gamma$ be a $[2]$-coloured graph and $C$ be a subset of $\cB^{\set{i,j}}(\Gamma)$, $i,j\in[2]$ relative to which we are going to do the partial duality. Let $k$ denote the unique element of
  $[2]\setminus\set{i,j}$. For all $t\in\set{i,j}$, we define
  \begin{itemize}
  \item $\bar\Gamma[C;tk]$ as the (possibly
    disconnected, not properly) edge coloured subgraph of $\Gamma$ made
of the cycles in $C$ and all the $tk$-cycles incident with $C$,
\item $\Gamma_{\!s}[C;tk]$ as the (possibly
    disconnected, properly) edge coloured
  graph obtained from $\bar\Gamma[C;tk]$ by contracting
  (in the sense of coloured graphs) all the $t$-edges not in $C$. 
  So every $tk$-path outside $C$ will be replaced by a single edge of colour $k$.
  \end{itemize}
\end{defn}
An example is given in \cref{f-SpecialSubgraphs}.
\begin{figure}[!htp]
  \centering
    \begin{subfigure}[c]{.5\linewidth}
      \centering
      \begin{tikzpicture}
      \def\u{2cm};
      \def\vsh{(.7*\u,.4*\u)};
      \draw[arete,draw=oneS] (0,0)-- node[lab,swap] {$1$} (\u,0);
      \draw[arete,draw=zeroS] (\u,0)--
      node[lab,swap,near end] {$0$} (\u,\u);
      \draw[arete,draw=oneS] (\u,\u)-- node[lab,swap] {$1$} (0,\u);
      \draw[arete,draw=zeroS] (0,\u)-- node[lab,swap] {$0$} (0,0);
      \draw[arete,draw=twoS] (\u,0) -- node[lab,swap] {$2$} +\vsh;
      \draw[arete,draw=twoS] (0,\u) -- node[lab] {$2$} +\vsh;
      \draw[arete,draw=twoS] (\u,\u) -- node[lab,swap] {$2$} +\vsh;
      \draw[arete, dashed, thick, shift={\vsh},draw=oneS] (0,0)--node[lab,swap] {$1$} (\u,0);
      \draw[arete, dashed, thick, shift={\vsh},draw=zeroS] (0,0)-- node[lab,swap,near start] {$0$} (0,\u);
      \draw[arete, dashed, thick,draw=twoS] (0,0)-- node[lab] {$2$} +\vsh;
      \draw[arete,shift={\vsh},draw=oneS] (0,\u) -- node[lab] {$1$} (\u,\u);
      \draw[arete,shift={\vsh},draw=zeroS] (\u,0) -- node[lab,swap] {$0$} (\u,\u);

      \node[vertex] at (0,0) {};
      \node[vertex] at (\u,0) {};
      \node[vertex] at (0,\u) {};
      \node[vertex] at (\u,\u) {};
      \node[vertex, shift={\vsh}] at (0,0) {};
      \node[vertex, shift={\vsh}] at (\u,0) {};
      \node[vertex, shift={\vsh}] at (0,\u) {};
      \node[vertex, shift={\vsh}] at (\u,\u) {};
      \end{tikzpicture}
      \caption{A coloured graph $\Gamma$; $C$ is the front $01$-cycle; $k=2$}
      \label{SpSubg1}
    \end{subfigure}\\
    \bigskip
    \begin{subfigure}[b]{.4\linewidth}
      \centering
      \begin{tikzpicture}
          \def\u{2cm};
          \def\vsh{(.7*\u,.4*\u)};
          \coordinate (p) at (.7*\u,1.4*\u);
          
          \draw[arete,draw=oneS] (0,0)-- node[lab,swap] {$1$} (\u,0);
          \draw[arete,draw=zeroS] (\u,0)--
          node[lab,swap] {$0$} (\u,\u);
          \draw[arete,draw=zeroS] (0,\u)-- node[lab,swap] {$0$} (0,0);
          \draw[arete,draw=twoS] (\u,0) -- node[lab,swap] {$2$} +\vsh;
          \draw[arete,draw=twoS] (0,\u) -- node[lab] {$2$} +\vsh;
          \draw[arete,draw=twoS] (\u,\u) -- node[lab] {$2$} +\vsh;
          \draw[arete,dashed,thick,draw=zeroS] \vsh -- node[lab,swap] {$0$} ({\vsh}
          |- 0,\u); 
        \draw[arete,draw=zeroS] (p) -- (p |- 0,\u);
        \draw[arete,draw=oneS] (\u,\u)-- node[lab,swap] {$1$} (0,\u);
          \draw[arete, thick, dashed,draw=twoS] (0,0)-- node[lab] {$2$} +\vsh;
          \draw[arete,shift={\vsh},draw=zeroS] (\u,0) -- node[lab,swap] {$0$} (\u,\u);

          \node[vertex] (c) at (0,0) {};
          \node[vertex] (a) at (\u,0) {};
          \node[vertex] (d) at (0,\u) {};
          \node[vertex] (b) at (\u,\u) {};
          \node[vertex, shift={\vsh}] at (0,0) {};
          \node[vertex, shift={\vsh}] at (\u,0) {};
          \node[vertex, shift={\vsh}] at (0,\u) {};
          \node[vertex, shift={\vsh}] at (\u,\u) {};
        \end{tikzpicture}
        \caption{$\bar\Gamma{[C;02]}$}
        \label{SpSubg2}
      \end{subfigure}
      \begin{subfigure}[b]{.4\linewidth}
        \centering
        \begin{tikzpicture}
          \def\u{2cm};
          
          \draw[arete,draw=oneS] (0,0)-- node[lab,swap] {$1$} (\u,0);
          \draw[arete,draw=zeroS] (\u,0)--
          node[lab] {$0$} (\u,\u);
          \draw[arete,draw=oneS] (\u,\u)-- node[lab,swap] {$1$}
          (0,\u);
          \draw[arete,draw=zeroS] (0,\u)-- node[lab] {$0$} (0,0);
          \draw[arete,draw=twoS] (a) [bend right=50] to node[lab,swap] {$2$} (b);
          \draw[arete,draw=twoS] (c) [bend left=50] to node[lab] {$2$} (d);
          
          \node[vertex] at (0,0) {};
          \node[vertex] at (\u,0) {};
          \node[vertex] at (0,\u) {};
          \node[vertex] at (\u,\u) {};
        \end{tikzpicture}
        \caption{$\Gamma_{\!s}{[C;02]}$}
        \label{SpSubg3}
      \end{subfigure}
  \caption{Special subgraphs of coloured graphs}
  \label{f-SpecialSubgraphs}
\end{figure}
\begin{lemma}
  \label{thm-Deltac}
  Let $\Gamma$, $C$, $k$ and $t$ be as in
  \cref{def-SpecialSubgraphs}. Then
  \begin{equation*}
    \begin{split}
      \Delta^{C}_{tk}\ \defi&\ 
      B^{\set{tk}}(\Gamma^{C})-B^{\set{tk}}(\Gamma)\\
      =&-2B^{\set{tk}}(\Gamma_{\!s}[C;tk])
      -\card C+n(\Gamma_{\!s}[C;tk])-\gamma(\Gamma_{\!s}[C;tk])+2k(\Gamma_{\!s}[C;tk]),
    \end{split}
  \end{equation*}
where $n(\Gamma_{\!s}[C;tk])$
is half of the number of vertices of $\Gamma_{\!s}[C;tk]$ 
and $k$
denotes the number of connected components.
\end{lemma}
\begin{proof}
  Let $\bar t$ be the unique element of $\set{i,j}\setminus\set{t}$. By definition,
  \begin{equation}
    \label{eq-Deltas}
    \begin{split}
      B^{\set{tk}}(\Gamma^{C})-B^{\set{tk}}(\Gamma)
      &=B^{\set{tk}}(\Gamma_{\!s}^C[C;tk])
      -B^{\set{tk}}(\Gamma_{\!s}[C;tk])\\
      &=B^{\set{\bar tk}}(\Gamma_{\!s}[C;tk])
      -B^{\set{tk}}(\Gamma_{\!s}[C;tk]).
    \end{split}
  \end{equation}
The surface corresponding to a $[2]$-coloured graph $\Gamma$ has Euler characteristic 
\begin{equation}
  \label{eq-EulerCharG}
  \chi(\Gamma)=2k(\Gamma)-\gamma(\Gamma)=B^{\set{ij}}(\Gamma)+B^{\set{ik}}(\Gamma)+B^{\set{jk}}(\Gamma)-n(\Gamma).
\end{equation}
For $\Gamma=\Gamma_{\!s}[C;tk]$ we have
$B^{\set{ij}}(\Gamma_{\!s}[C;tk])=\card C$ and \cref{eq-EulerCharG} gives
$$B^{\set{\bar tk}}(\Gamma_{\!s}[C;tk])=
-B^{\set{tk}}(\Gamma_{\!s}[C;tk])
      -\card C+n(\Gamma_{\!s}[C;tk])-\gamma(\Gamma_{\!s}[C;tk])+2k(\Gamma_{\!s}[C;tk]).
$$
Plugging it into \cref{eq-Deltas}, one gets the desired result.
\end{proof}
\begin{thm}
  \label{thm-GenusFormula}
  Let $\Gamma$ be a $[2]$-coloured graph and $C$ be a subset of $\cB^{\set{i,j}}(\Gamma)$. Let $k$ be the unique element of $[2]\setminus\set{i,j}$. Then,
  \begin{equation*}
    \gamma(\Gamma^{C})-\gamma(\Gamma)=
    -\Delta^{C}_{ik}(\Gamma)
     -\Delta^{C}_{jk}(\Gamma)
  \end{equation*}
where $\Delta^{C}_{tk}$ is given by \cref{thm-Deltac}.
\end{thm}
\begin{proof}
  One simply uses \cref{eq-EulerCharG} and remarks that
  $k(\Gamma^{C})=k(\Gamma)$,
  $n(\Gamma^{C})=n(\Gamma)$,
  and $B^{\set{ij}}(\Gamma^{C})=B^{\set{ij}}(\Gamma)$.
\end{proof}

\begin{rem}
  \Cref{thm-GenusFormula} alows to derive bounds on $\gamma(\Gamma^{C})-\gamma(\Gamma)$. For any $t\in\set{i,j}$ and any coloured graph $\Gamma$, the number of
  $tk$-cycle in $\Gamma_{\!s}[C;tk]$,
  $B^{\set{tk}}(\Gamma_{\!s}[C;tk])$, lies between $1$ and
  $n(\Gamma_{\!s}[C;tk])=\tfrac 12\sum_{c\in
    C}\text{length}(c)$. 
This gives
  \begin{equation*}
    |\gamma(\Gamma^{C})-\gamma(\Gamma)|\les\sum_{c\in C}\del{\text{length}(c)-2}.
  \end{equation*}
  This bound is optimal and \cref{f-OptimalGenusChange} shows an
  example where it is reached.
\end{rem}
\begin{figure}[!htp]
  \centering
  \begin{tikzpicture}[scale=2, rotate around x=-90, rotate around z=0]
    \foreach \i in {0,...,5}
    {
      \coordinate (\i) at ({cos(60/2+60*\i)},{sin(60/2+60*\i)},0);
      \coordinate (s\i) at ({cos(60/2+60*\i)},{sin(60/2+60*\i)},1);
    }
    \draw[arete, dashed, thick,draw=twoS] (1)-- node[lab,near end] {$2$} (s1)
    (2)-- node[lab,swap,near end] {$2$} (s2);
    \draw[arete,draw=twoS] (0)-- node[lab,swap] {$2$} (s0) (3)-- node[lab] {$2$}
    (s3) (4)-- node[lab,swap] {$2$} (s4) (5)-- node[lab,swap,near end] {$2$} (s5);
    \draw[arete,draw=twoS] (0)-- node[lab,swap] {$2$} (s0);
    \draw[arete, dashed, thick,draw=zeroS] (0)-- node[lab, swap,near start] {$0$} (1);
    \draw[arete, dashed, thick,draw=oneS] (1)--
    node[lab, swap,near start] {$1$} (2);
    \draw[arete, dashed, thick,draw=zeroS] (2)-- node[lab, swap] {$0$} (3);
    \draw[arete,draw=zeroS] (s0)-- node[lab, swap] {$0$} (s1);
    \draw[arete,draw=oneS] (s1)--
    node[lab, swap] {$1$} (s2);
    \draw[arete,draw=zeroS] (s2)-- node[lab, swap] {$0$} (s3);
    \draw[arete,draw=oneS] (s3)--
    node[lab, near end] {$1$} (s4);
    \draw[arete,draw=zeroS] (s4)-- node[lab] {$0$} (s5);
    \draw[arete,draw=oneS] (s5)-- node[lab] {$1$}
    (s0);
    \draw[arete,draw=oneS] (3)-- node[lab,swap] {$1$} (4);
    \draw[arete,draw=zeroS] (4)-- node[lab,swap] {$0$}
    (5);
    \draw[arete,draw=oneS] (5)-- node[lab,swap] {$1$} (0);
    \foreach \i in {0,...,5}
    {
      \node[vertex] at (\i) {};
      \node[vertex] at (s\i) {};
    }
  \end{tikzpicture}
  \caption{A coloured graph $\Gamma$ such that
    $g(\Gamma^C)-g(\Gamma)=2$, $C$ is anyone of its $01$-cycles.}
\label{f-OptimalGenusChange}
\end{figure}

Given the bijection between hypermaps and $[2]$-coloured graphs,
ribbon graphs are $[2]$-coloured graphs, the $02$-cycles of which all
have length four. \Cref{thm-GenusFormula} then applies and allows to
quantify the change of topology of ribbon graphs under partial
duality:
\begin{cor}
  \label{thm-GenusDualRibbons}
  Let $G$ be a ribbon graph and $E'$ be a subset of its edges. Let
  $G[E']$ be the subribbon graph of $G$ induced by $E'$ and
  $G^{*}[E']$ be the subribbon graph of its Euler-Poincaré dual
  $G^{*}$ induced by $E'$. Then
  \begin{equation*}
    \tfrac
    12\big(\gamma(G^{E'})-\gamma(G)\big)=v(G[E'])+v(G^{*}[E'])-\card
    E'-\tfrac 12\chi(G[E'])-\tfrac 12\chi(G^{*}[E']).
  \end{equation*}
\end{cor}
\begin{proof}
  With our conventions, $C=E'$ is a subset of $02$-cycles,
  vertices are $12$-cycles and faces are $01$-cycles. Thus,
  $\Gamma_{\!s}[E';12]=G[E']$, $\Gamma_{\!s}[E';01]=(G^{*}[E'])^{*}$,
  $n(\Gamma_{\!s}[E';1t])=2\card E'$, and the \namecref{thm-GenusDualRibbons} follows.
\end{proof}

In particular, for partial duality of ribbon graphs relative to a single edge, 
$\card E'=1$, this \namecref{thm-GenusDualRibbons} immediately gives the results of 
\cite[Table 1.1]{Gross2020aa} which were recently used in \cite{Chmutov2021aa} to prove one of the conjectures from \cite{Gross2020aa}.

\section{Directions of future research} \label{sec-future}

\begin{itemize}
\item The paper \cite{Gross2020aa} contains several interesting conjectures about partial-dual genus distribution polynomial for ribbon graphs. One of them was recently proved in  \cite{Chmutov2021aa}. The definition of this polynomial works for hypermaps as well. It would be interesting to formulate and prove them for hypermaps.

\item Maps (ribbon graphs) provide a special class of $\Delta$-matroids (Lagrangian matroids) \cite{Chun2014}.
Are there any matroid type structure underlying the concept of hypermaps? Can the general Coxeter matroids be obtained from hypermaps?

\item It would be interesting to study higher dimensional partial duality concept as it outlined in Section \ref{ssec-pdual-cgr}. In particular, is it true that a partial dual to a $[d]$-coloured graph corresponding to a manifold is also a manifold? 
\end{itemize}

\newpage
\printbibliography

\contactrule
\contactSChmutov
\contactFVT

\end{document}